\documentclass[11pt]{amsart}

\usepackage{amssymb, amsthm, amsmath, gensymb}
\usepackage{graphicx, comment}
\usepackage[small]{caption}
\usepackage{subcaption}
\usepackage{epsfig}
\usepackage{tikz, pgfplots, float}

\newcommand{\later}[1]{}
\newcommand{\old}[1]{}

\usepackage{amsfonts}

\usepackage[utf8]{inputenc}
\usepackage{fullpage}
\usepackage{framed}
\usepackage{multirow}
\usepackage{enumerate}
\usepackage{url}
\usepackage[breaklinks]{hyperref}
\usepackage{cleveref}
\hypersetup{
	colorlinks = true, 
	urlcolor = cyan, 
	linkcolor = teal, 
	citecolor = cyan 
}

\newcommand{\LA}{\mathop{}\!\mathrm{La^*}}

\newcommand{\LAw}{\mathop{}\!\mathrm{La}}

\newcommand{\arm}{\mathop{}\!\mathrm{ar^*_m}}
\newcommand{\arw}{\mathop{}\!\mathrm{ar_m}}
\newtheorem{thm}{Theorem}[section]
\newtheorem{definition}[thm]{Definition}
\newtheorem{lem}[thm]{Lemma}

\newtheorem{cor}[thm]{Corollary}

\newtheorem{prop}[thm]{Proposition}
\theoremstyle{definition}
\newtheorem{rem}[thm]{Remark}

\newtheorem{obs}[thm]{Observation}

\newtheorem{conjecture}[thm]{Conjecture}

\newcommand{\cC}{{\mathcal C}}
\newcommand{\cQ}{{\mathcal Q}}

\newcommand{\cF}{{\mathcal F}}

\newcommand{\cG}{{\mathcal G}}
\newcommand{\cD}{\boldsymbol{\mathcal D}}
\newcommand{\cM}{{\mathcal M}}
\newcommand{\cH}{{\mathcal H}}

\newcommand{\cA}{{\mathcal A}}

\newcommand{\cT}{{\mathcal T}}

\newcommand{\cU}{{\mathcal U}}

\title{Maximal anti-Ramsey problems for posets}
\author{Binlong Li}
\address{School of Mathematics and Statistics, Xi’an–Budapest Joint Research Center for Combinatorics, Northwestern Polytechnical University, Xi’an 710129, China.} 
\email{binlongli@nwpu.edu.cn}

\author{Bal\'azs Patk\'os}
\address{HUN-REN Alfr\'ed R\'enyi Institute of Mathematics, } 
\email{patkos@renyi.hu}

\author{Changxin Wang}
\address{School of Mathematics and Statistics, Xi’an–Budapest Joint Research Center for Combinatorics, Northwestern Polytechnical University, Xi’an 710129, China.} 
\email{simonang@mail.nwpu.edu.cn}

\date{}

\begin{document}

\begin{abstract}
    We study the forbidden poset analog of the maximal anti-Ramsey problem introduced for graphs by Burr, Erd\H os, Graham, and S\'os. For integers $m\le 2^n$ and poset $P=(P,\preceq)$, we introduce $\arw(n,m,P)$ (and $\arm(n,m,P)$) to denote the minimum integer $k$ such that there exists a family $\mathcal{F}\subseteq 2^{[n]}$ with $|\cF|=m$ and a coloring $\cF\rightarrow [k]$ with all weak (strong) copies of $P$ being rainbow. As long as there exist $P$-free families of size $m$, these parameters equal 1. It is known that the largest size $\LAw(n,P)$ ($\LA(n,P)$) of weak (strong) $P$-free families has order of magnitude $\Theta(\binom{n}{\lfloor n/2\rfloor})$ unless $P$ is the antichain $A_k$ on $k$ elements. In this paper we study $\arw(n,m,P)$ and $\arm(n,m,P)$ in two regimes of $m$. We determine the asymptotics of these parameters for all posets $P$ when $m=2^n$. We also consider the case $m=\Theta(\binom{n}{\lfloor n/2\rfloor})$. It is shown that for any connected poset $P$ and integer $k$, there exist integers $m_{P,k}$ and $m^*_{P,k}$ such that to color the middle $k$ layers of the Boolean lattice with all weak or strong copies of $P$ being rainbow, one needs $\Theta(n^{m_{P,k}})$ or $\Theta(n^{m^*_{P,k}})$ colors. For tree posets $T$, one has $m_{T,k}=m^*_{T,k}$. We conjecture that for any tree poset $T$, and positive real $\varepsilon$, $\arw(n,m,T),\arm(n,m,T)=\Omega(n^{m_{T,k}})$ holds provided $m\ge (k-1+\varepsilon)\binom{n}{\lfloor n/2\rfloor}$. We prove our conjecture on $\arw(n,m,T)$ for an infinite class of tree posets.
\end{abstract}

\maketitle

\section{Introduction}

We use standard notation: $[n]$ denotes the set of the first $n$ positive integers. For a set $S$, its power set is denoted by $2^S$ and we write $\binom{S}{k}=\{T\subseteq S: |T|=k\}$.

\medskip

In extremal combinatorics, Tur\'an type problems ask for the maximum size that a combinatorial structure (graph, hypergraph, etc) can have without containing a forbidden substructure. In extremal finite set theory, the definition of a forbidden containment pattern was introduced by Katona and Tarj\'an \cite{KT}: a family $\cG$ of sets is a \textit{weak copy} of the poset $(P,\preceq)$, if there exists a bijection $\iota:P\rightarrow \cG$ such that $p\preceq q$ implies $\iota(p) \subseteq \iota(q)$. $\cG$ is said to be a \textit{strong copy} of $(P,\preceq)$ if the bijection $\iota$ satisfies $p\preceq q$ if and only if $\iota(p)\subseteq \iota(q)$. A family $\cF$ is weak (strong) $P$-free if it does not contain any weak (strong) copies of $P$. The extremal problem, the so-called forbidden subposet problem, of finding $\LAw(n,P)$ ($\LA(n,P)$) the maximum size of a weak (strong) $P$-free family $\cF\subseteq 2^{[n]}$ has attracted lots of attention: for surveys see \cite{AMP,GLsurv} and Chapter 7 of \cite{GP}.

A combinatorial object is \textit{rainbow} with respect to a coloring $c$ if all its elements (edges of a graph, sets of a set family, etc) receive distinct colors. Many extremal graph theoretic questions have been studied that ask for conditions that imply the existence of a rainbow copy of a graph $H$. Recently, several papers \cite{KM,KMOWY,P1,P2} have considered forbidden subposet analogs of these graph theory rainbow problems. The present manuscript addresses the forbidden subposet analog of the following version introduced by Burr, Erd\H os, Graham, and S\'os \cite{BEGS} (for more recent development, see \cite{BCM,LNX,SS,Y}): what is the minimum number $\chi_s(n,e,H)$ of colors that is needed to color some $n$-vertex graph $G$ with $e$ edges such that all copies of $H$ are rainbow. Here are the main definitions for the present paper.

\begin{definition}
For any poset $P$ and integers $n,m$ with $1<m\le 2^n$, we define the \textit{maximal anti-Ramsey number} $\arw(n,m,P)$ to be the minimum integer $k$ of colors such that there exist $\cF\subseteq 2^{[n]}$ with $|\cF|=m$ and a coloring $\cF\rightarrow [k]$ with all weak copies of $P$ in $\cF$ being rainbow.  

For any poset $P$ and integers $n,m$ with $1<m\le 2^n$, we define the \textit{maximal anti-Ramsey number} $\arm(n,m,P)$ to be the minimum integer $k$ of colors such that there exist $\cF\subseteq 2^{[n]}$ with $|\cF|=m$ and a coloring $\cF\rightarrow [k]$ with all strong copies of $P$ in $\cF$ being rainbow. 

Clearly, when $m\leq \LAw(n,P)$ or $\LA(n,P)$, we have $\arw(n,m,P)$ or $\arm(n,m,P)=1$, respectively. Hence, in the following content, we only consider the case where $m> \LAw(n,P)$ or $\LA(n,P)$.
\end{definition}

Observe that for any poset $P$ and integer $m$, we have $\arm(n,m,P)\le \arw(n,m,P)$. Indeed, if $\cF\subseteq 2^{[n]}$ with $|\cF|=m$ and $c:\cF\rightarrow [k]$ with $k=\arw(n,m,P)$ is a coloring with all weak copies of $P$ being rainbow, then, as all strong copies are also weak copies, all strong copies of $P$ are rainbow and thus $\arm(n,m,P)\le k$ as claimed.

\medskip

The two basic posets are chains and antichains, denoted $C_k$ and $A_k$ respectively, with the index indicating the number of elements in the poset. Note that every weak copy of $C_k$ is also a strong copy, thus $\arm(n,m,C_k)=\arw(n,m,C_k)$. Any family of $k$ sets forms a weak copy of $A_k$, so $\arw(n,m,A_k)=m$ for all $m\ge k$. Let $\Sigma(n,k)$ denote the sum of the $k$ largest coefficients of order $n$, i.e. $\Sigma(n,k)=\sum_{i=1}^{k}\binom{n}{\lfloor\frac{n-k}{2}\rfloor+i}$. Using the notations above, Erd\H os's theorem \cite{E} states $\LA(n,C_{k+1})=\LAw(n,C_{k+1})=\Sigma(n,k)$ for any $n\ge k$.

\begin{prop}
    For any $k\le \ell$ and $\Sigma(n,\ell)<m\le \Sigma(n,\ell+1)$, we have $\arm(n,m,C_{k+1})=\ell+1$. 
\end{prop}

\begin{proof}
    To obtain the upper bound, consider a subfamily  $\cF$ of size $m$ of the $\ell+1$ middle layers of $2^{[n]}$ and color sets in $\cF$ by their size.

    To see the lower bound, by Erd\H os's theorem, any $\cF\subseteq 2^{[n]}$ of size $m$ contains a chain $\cC$ of length $\ell+1$. In order to have all copies of $C_{k+1}$ rainbow, all sets of $\cC$ need to receive distinct colors.
\end{proof}

\begin{prop}
    For any $k\le \ell+1\le \binom{n}{\lfloor n/2\rfloor}$ and $\LA(n,A_{\ell+1})<m\le \LA(n,A_{\ell+2})$, we have $\arm(n,m,A_k)=\ell+1$.
\end{prop}

\begin{proof}
Any family $\cF$ of size $m$ contains a strong copy of $A_{\ell+1}$ and its sets need to be colored by distinct colors. This shows $\arm(n,m,A_k)\ge \ell+1$. On the other hand, there exists a family $\cF\subseteq 2^{[n]}$ of size $m$ containing no $A_{\ell+2}$. By Dilworth's theorem $\cF$ can be partitioned into $\ell+1$ chains. Coloring the sets in $\cF$ according to the chain they belong to yields a coloring showing $\arm(n,m,A_k)\le \ell+1$.
\end{proof}

For a poset $P$ and a family $\cF$ of sets, let $G^*_P(\cF)$ ($G_P(\cF)$) denote the graph on vertex set $\cF$ with $F,G\in \cF$ being joined by an edge if and only if there exists a strong (weak) copy of $P$ in $\cF$ containing both $F$ and $G$. Clearly, $G^*_P(\cF)$ is a subgraph of $G_P(\cF)$ for all posets $P$ and families $\cF$. We will write $G^*_P(n)$ and $G_P(n)$ to denote $G^*_P(2^{[n]})$ and $G_P(2^{[n]})$, respectively. Also, $G_P[\cF]$ (and $G_P^*[\cF]$) denotes the subgraph of $G_P(n)$ ($G^*_P(n)$) induced on $\cF$. We have $G_P(\cF)\subseteq G_P[\cF]$ and $G_P^*(\cF)\subseteq G_P^*[\cF]$ as it might happen that all weak or strong copies of $P$ containing $F,F'$ also contain at least one set from $2^{[n]} \setminus \cF$.

\begin{prop}\label{chromatic}
    For any poset $P$ and integers $m,n$ we have $\arm(n,m,P)=\min\{\chi(G^*_P(\cF)): |\cF|=m,\cF\subseteq 2^{[n]}\}$ and $\arw(n,m,P)=\min\{\chi(G_P(\cF)): |\cF|=m,\cF\subseteq 2^{[n]}\}$. 
\end{prop}

\begin{proof}
    We only consider the case of strong copies as the proof for weak copies is analogous.
    First, let $\cF\subseteq 2^{[n]}$ be a family of size $m$ such that the chromatic number of $G^*_P(\cF)$ is minimum possible. Then a proper coloring of $G^*_P(\cF)$ gives a coloring of $\cF$ such that all strong copies of $P$ are rainbow. Indeed, by definition, the sets of any strong copy of $P$ form a clique in $G^*_P(\cF)$, so in a proper coloring they should receive distinct colors. This shows   $\arm(n,m,P)\le \min\{\chi(G^*_P(\cF): |\cF|=m,\cF\subseteq 2^{[n]}\}$.

    Conversely, let $\cF\subseteq 2^{[n]}$ of size $m$ and $c:\cF\rightarrow [k]$ with $k=\arm(n,m,P)$ and all strong copies of $P$ being rainbow with respect to $c$. Then $c$ is a proper coloring of $G^*_P(\cF)$ as if $F,G\in \cF$ are contained in a strong copy of $P$ in $\cF$, then this strong copy is rainbow, so $c(F)\neq c(G)$. This shows  $\arm(n,m,P)\ge \min\{\chi(G^*_P(\cF): |\cF|=m,\cF\subseteq 2^{[n]}\}$. 
\end{proof}

\subsection{Our results}

If $P$ is not an antichain, then the family of all sets of size $\lfloor n/2\rfloor$ shows $\LAw(n,P),\LA(n,P)\ge \binom{n}{\lfloor n/2\rfloor}$. As a chain of size $|P|$ is a weak copy of $P$, Erd\H os's result on $\LA(n,C_{k+1})$ shows $\LAw(n,P)=\Theta(\binom{n}{\lfloor n/2\rfloor})$. The result  $\LA(n,P)=\Theta(\binom{n}{\lfloor n/2\rfloor})$ was proved by Methuku and P\'alv\"olgyi \cite{MP}. We will address determining $\arw(n,m,P)$ and $\arm(n,m,P)$ at the two extreme values of $m$: when $m=2^n$ is maximum possible and so the only family of size $m$ is $2^{[n]}$, and when $m=\Theta(\binom{n}{\lfloor n/2\rfloor})$ is minimum possible.

\medskip

To state our results, we need a couple of definitions.
Let $U_P(x)=\{y\in P: x\prec y\}, D_P(x)=\{y\in P: y\prec x\}$, where $\prec$ means $x\preceq y$ and $x\neq y$. Write $U_P[x]=U_P(x)\cup \{x\}$ and $D_P[x]=D_P(x)\cup \{x\}$. 

A pair $p,q\in P$ of incomparable elements is \textit{weak bad} if any of the following hold:
\begin{itemize}
    \item A pair $(p,q)$ is weak upper bad: if $|U_P(p)\cap U_P(q)|\ge 2$.
    \item A pair $(p,q)$ is weak lower bad: if $|D_P(p)\cap D_P(q)|\ge 2$.
\end{itemize}
An incomparable pair that is not weak bad is \textit{weak good}. Basically, $p,q$ being weak bad means that for any complement pair of sets $F,\bar{F}=[n]\setminus F$, there is no weak copy of $P$ with $F,\bar{F}$ playing the role of $p$ and $q$. So if all incomparable pairs are weak bad, then $F\bar{F}$ is not an edge in $G_P(n)$ and therefore $F$ and $\bar{F}$ can be colored with the same color. Our first theorem states that this is the only factor that matters when determining the asymptotics of $\arw(n,2^n,P)$.

\begin{thm}\label{weakm=2n}
    For any $P\neq A_k,C_k$, we have the following:
    \begin{itemize}
        \item 
        If $P$ contains a weak good incomparable pair, then $\arw(n,2^n,P)=(1+o(1))2^n$.
        \item 
        If all incomparable pairs of $P$ are weak bad, then $\arw(n,2^n,P)=(\frac{1}{2}+o(1))2^n$.
    \end{itemize}
\end{thm}

The situation is much more interesting for strong copies of $P$ and for $\arm(n,2^n,P)$. We need to alter the definition of being bad.

A pair $p,q\in P$ of incomparable elements is \textit{strong bad} if any of the following hold:
\begin{itemize}
    \item A pair $(p,q)$ is strong upper bad: if $|U_P(p)\cap U_P(q)|\ge 2$ or $|U_P(p)\cap U_P(q)|=1$ and $P$ does not have a largest element.
    \item A pair $(p,q)$ is strong lower bad: if $|D_P(p)\cap D_P(q)|\ge 2$ or $|D_P(p)\cap D_P(q)|=1$ and $P$ does not have a smallest element.
\end{itemize}

An incomparable pair is \textit{strong good} if it is not strong bad. A poset $P$ is \textit{incomparable good} if it contains a strong good incomparable pair, otherwise $P$ is incomparable bad.  

Let us write $\cM_n=\{F\in 2^{[n]}:||F|-n/2|\le n^{2/3}\}$ to denote the subfamily of $2^{[n]}$ consisting of all sets whose size is close to $n/2$. By Chernoff's inequality, we have $|2^{[n]}\setminus \cM_n|=o(2^n)$. Clearly, writing $\alpha(G)$ for the independence number of a graph $G$, we have 
$$\arm(n,2^n,P)=\chi(G_P^*(n))\ge \chi(G_P^*[\cM_n]) \ge \frac{|\cM_n|}{\alpha(G_P^*[\cM_n])}=\frac{(1-o(1))2^n}{\alpha(G_P^*[\cM_n])}.$$

Our next result states that asymptotically we have equality in the above chain of inequalities.

\begin{thm}\label{strongm=2n}
    For any poset $P\neq A_k,C_k$, we have
    $$\arm(n,2^n,P)=\frac{(1-o(1))2^n}{\alpha(G_P^*[\cM_n])}.$$
    Furthermore, $\alpha(G_P^*[\cM_n])$ can take any positive integer as its value.
\end{thm}

We shall determine the value of $\alpha(G_P^*[\cM_n])$ for all posets, but we need more categories to describe our result. A poset is \textit{incomparable upper (lower) bad} if all incomparable pairs are strong upper (lower) bad. Finally, $P$ is \textit{incomparable mixed bad} if every incomparable pair of $P$
is strong bad, but $P$ is neither incomparable upper bad nor incomparable lower bad. 

For a poset $P\neq A_k,C_k$, define 
$$\mathcal{E}(P)=\{(N,\iota):\text{ $N$ is a positive integer and $\iota :P\rightarrow 2^{[N]}$ is a strong embedding}\}.$$
Let $$\sigma(P)=\min_{\substack{(N,\iota)\in \mathcal{E}(P)\\ p\prec q\in P}}\{|F'\backslash F|: \text{$\iota(q)=F'$ and $\iota(p)=F$}\}.$$ 
We write $p\perp q$ to denote the fact that $p,q$ are incomparable in $P$. For mixed bad posets $P$ we define
$$\lambda^-(P)=\min_{\substack{(N,\iota)\in \mathcal{E}(P)\\ p\perp q\in P}}\{|F\cap F'|: \text{$\iota(p)=F$,  $\iota(q)=F'$ and $F,F'$ satisfying $F\cup F'=[N]$} \}$$ and 
$$\lambda^+(P)=\min_{\substack{(N,\iota)\in \mathcal{E}(P)\\ p\perp q\in P}}\{N-|F\cup F'|: \text{$\iota(p)=F$,  $\iota(q)=F'$ and $F,F'$ satisfying $F\cap F'=\emptyset$} \}.$$ 

\begin{thm}\label{dichotomy}
    For any poset $P\neq A_k,C_k$ and $n$  sufficiently large, an independent set in $G_P^*[\cM_n]$ is either a chain of sets or a union of two chains. Furthermore, we have the following:
    \begin{enumerate}
        \item 
        If $P$ is incomparable-good, then $\alpha(G_P^*[\cM_n])=\sigma(P)$, where $\sigma(P)$ can be arbitrary large.
        \item 
        If $P$ is incomparable-upper bad or incomparable-lower bad, then $\alpha(G_P^*[\cM_n])=2$.
        \item 
        If $P$ is incomparable-mixed bad, then 
        $\alpha(G_P^*[\cM_n])=\max\{\sigma(P),\min\{\lambda^-(P)+\lambda^+(P),2\sigma(P)\}\}$.
    \end{enumerate}
\end{thm}

Now we turn our attention to the case $m=\Theta(\binom{n}{\lfloor n/2\rfloor})$. The \textit{oriented Hasse diagram} $\overrightarrow{H}(P)$ of a poset $P$ is an oriented graph with vertex set $P$ where $\overrightarrow{pq}$ is an arc if $q$ covers $p$, i.e. $p\leqslant_P q$ and there is no $z\neq p,q$ with $p\leqslant_P z\leqslant_P q$. We obtain the \textit{(unoriented) Hasse diagram} $H(P)$ of $P$ from $\overrightarrow{H}(P)$ by dropping the orientation of all arcs. A poset is a \textit{tree poset} if $H(P)$ is a tree. Bukh \cite{B} proved $\LAw(n,T)=(h(T)-1+o(1))\binom{n}{\lfloor n/2\rfloor}$, where $h(T)$ is the height of $T$, the largest integer $k$ for which $T$ contains a copy of $C_k$. Boehnlein and Jiang \cite{BJ} strengthened this result to $\LA(n,T)=(h(T)-1+o(1))\binom{n}{\lfloor n/2\rfloor}$. The lower bounds in both of these results are given by the middle $h(T)-1$ layers of $2^{[n]}$. We write $\cM_{n,k}=\cup_{i=1}^{k}\binom{[n]}{\lfloor \frac{n-k}{2}\rfloor+i}$ to denote the family of the $k$ middle layers.

\begin{prop}\label{exponent}
    For any connected poset $P\neq C_k$ and integer $k$, there exist non-negative integers $m_{P,k}$ and $m^*_{P,k}$ such that $$\chi(G_P(\cM_{n,k}))=\Theta(n^{m_{P,k}}), \quad \chi(G^*_P(\cM_{n,k}))=\Theta(n^{m^*_{P,k}}).$$
\end{prop}

\begin{conjecture}\label{conj}
   For any tree poset $T\neq C_k$, integer $k\ge h(T)$, and positive real $\varepsilon>0$, there exists a constant $\delta>0$ such that $\arw(n,(k-1+\varepsilon)\binom{n}{\lfloor \frac{n}{2}\rfloor},T)\ge \arm(n,(k-1+\varepsilon)\binom{n}{\lfloor \frac{n}{2}\rfloor},T)\ge \delta n^{m_{T,k}}$.
\end{conjecture}

We will prove some special cases of Conjecture \ref{conj}. We say that a tree poset $T$ is \textit{upward monotone} if it has a smallest element, and $T$ is \textit{downward monotone} if it has a largest element.  $T$ is \textit{two-way monotone} if it contains an element $x$ that is comparable to all elements $y\in T$. If $T$ is two-way monotone, then the elements comparable to all elements form a (possibly one element) chain $x_1\prec x_2\prec \dots \prec x_k$ such that $T=D_T(x_1)\cup U_T(x_k)\cup \{x_j:1\le j \le k\}$ and $D_T(x_1)$ is downward monotone and $U_T(x_k)$ is upward monotone. An upward monotone tree $T$ is \textit{branching at level $i$} if there exist two distinct leaves $x,y$ such $\min\{d(r,z): z$ is on the path between $x$ and $y$ in $H(T)\}=i$, where $r$ is the smallest element of $T$. An upward monotone tree is \textit{branching} if whenever it branches at level $i<h(T)-1$, then it branches at level $i+1$. A downward monotone tree poset $T$ is branching if its dual is branching. A two-way monotone tree poset $T$ is branching if for the chain $x_1\prec x_2\prec \dots \prec x_k$ of elements comparable to all other elements of $T$, we have that out of $D_T(x_1)$ and $U_T(x_k)$ the one with non-smaller height is branching (if these are of the same height, then at least one of them is branching).

\bigskip

\begin{figure}
    \begin{tikzpicture}
      \node[draw, circle] (0) at (6,-2) {};
    \node[draw, circle] (a1) at (7,0) {};
    \node[draw, circle] (a2) at (5,0) {};
    \node[draw, circle] (b2) at (4,2) {};
    \node[draw, circle] (b1) at (8,2) {};
    \node (T) at (8,-2) {$T$};

    \node[draw, circle] (0') at (12,-2) {};
    \node[draw, circle] (a1') at (13,0) {};
    \node[draw, circle] (a2') at (11,0) {};
    \node[draw, circle] (b2') at (10,2) {};
    \node[draw, circle] (b1') at (14,2) {};
    \node[draw, circle] (b3') at (12,2) {};
    \node (T') at (14,-2) {$T'$};

    \node[draw, circle] (0'') at (18,0) {$x$};
    \node[draw, circle] (a1'') at (19,1) {};
    \node[draw, circle] (a2'') at (17,1) {};
    \node[draw, circle] (b2'') at (16,2) {};
    \node[draw, circle] (b1'') at (20,2) {};
    \node (T'') at (21,-2) {$T''$};

    \node[draw, circle] (a1''') at (17,-1) {};
    \node[draw, circle] (a2''') at (19,-1) {};
    \node[draw, circle] (b2''') at (16,-2) {};
    \node[draw, circle] (b1''') at (18,-2) {};
    \node[draw, circle] (b3''') at (20,-2) {};
    
\draw (0) -- (a1);
\draw (0) -- (a2);
\draw (b1) -- (a1);
\draw (b2) -- (a2);

\draw (0') -- (a1');
\draw (0') -- (a2');
\draw (b1') -- (a1');
\draw (b2') -- (a2');
\draw (b3') -- (a1');

\draw (0'') -- (a1'');
\draw (0'') -- (a2'');
\draw (b1'') -- (a1'');
\draw (b2'') -- (a2'');

\draw (0'') -- (a1''');
\draw (0'') -- (a2''');
\draw (b1''') -- (a1''');
\draw (b3''') -- (a2''');
\draw (b2''') -- (a1''');
\end{tikzpicture}
\caption{$T$ is not branching as it branches at level 0, but not at level 1. $T'$ is branching. $T''$ is also branching as $D_{T''}(x)$ and $U_{T''}(x)$ both have height 2 and the former is branching.}
\end{figure}
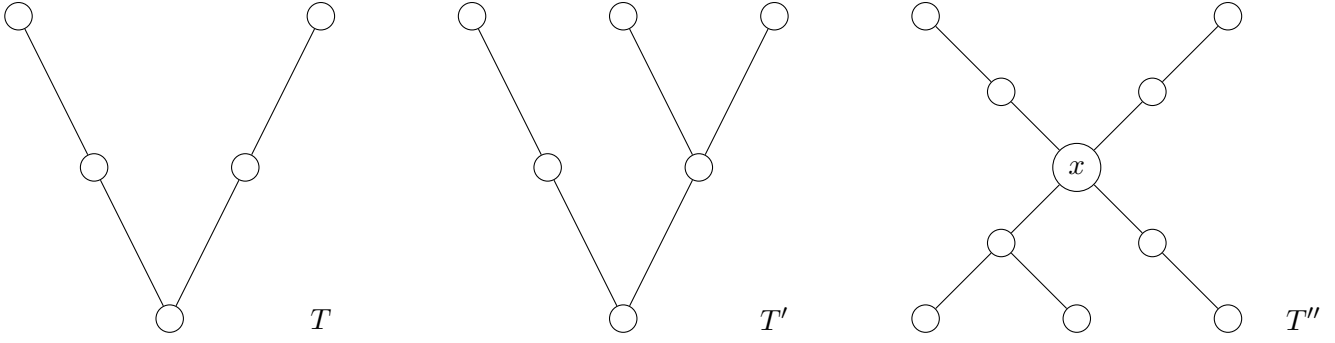

\begin{thm}\label{middletree}
    \begin{enumerate}
        \item 
     For any two-way monotone, branching tree poset $T$, integer $k\ge h(T)$, and positive real $\varepsilon>0$, there exists a constant $\delta>0$ such that $\arw(n,(k-1+\varepsilon)\binom{n}{\lfloor \frac{n}{2}\rfloor},T)\ge \delta n^{m_{T,k}}$.
     \item 
     For any tree poset of height 2,  and positive real $\varepsilon>0$, there exists a constant $\delta>0$ such that $\arw(n,(1+\varepsilon)\binom{n}{\lfloor \frac{n}{2}\rfloor},T)\ge \delta n^{m_{T,2}}$. 
     \end{enumerate}
\end{thm}

Our main tool will be two embedding results that are of independent interest and partially confirm a conjecture of Patk\'os and Treglown. 

\begin{definition}
    Let $T$ be a tree poset and $x\in T$, and let $d\ge 2$ be a positive integer. The \textit{$d$-blow-up $T(x,d)$ rooted at $x$} is the tree poset whose Hasse diagram is defined as follows: for each $u \in T$, if $u$
    is at distance $\rho$ from $x$ in  the undirected Hasse diagram of $T$, then $u$ is replaced with $d^{\rho}$ elements $u^1,u^2,\dots, u^{d^{\rho}}$; furthermore, if $uv$ is an edge of the Hasse diagram of $T$ with $v$ being at distance $\rho-1$ from $x$ in $T$, then the $u^i$s are partitioned into $d^{\rho-1}$ pairwise disjoint sets $U^1,U^2,\dots, U^{d^{\rho-1}}$ each of size $d$, and for every $j \in [d^{\rho-1}]$,
    $v^j$ is joined to all members of $U^j$. The orientation of all such edges is the same as that of $uv$.
\end{definition}

Throughout the paper, we will write $T(x,\alpha n)$ instead of $T(x,\lfloor \alpha n\rfloor)$.

\begin{conjecture}[Conjecture 1.3 in \cite{PT}]\label{conjPT}
    Let $T$ be any tree poset of height $h$. Given any $\varepsilon>0$ there exists $\delta >0$ such that the following holds for any  large enough $n$ and any $x \in P$.
   Consider the $\delta n$-blow-up $T({x, \delta n})$ of $T$  rooted at $x$. If $\mathcal F \subseteq 2^{[n]}$ has size at least $ (h-1+\varepsilon ) \binom{n}{\lfloor n/2 \rfloor}$, then
    $\mathcal F$ contains a weak copy of $T({x, \delta n})$.
\end{conjecture}

We will prove special cases of Conjecture \ref{conjPT}. Let us remark that a weaker, but general result, the existence of a weak copy of $T(x,\log n)$ was proved in \cite{BGW}.

\begin{thm}\label{blowup}
\begin{enumerate}
    \item 
    For any two-way monotone (not necessarily branching) tree poset $T$ and $x\in T$ comparable to all elements of $T$ and $\varepsilon>0$, there exists $\delta >0$ such that for $n$ large enough any family $\cF\subseteq 2^{[n]}$ with $|\cF|\ge (h(T)-1+\varepsilon)\binom{n}{\lfloor \frac{n}{2}\rfloor}$ contains a weak copy $T(x,\delta n)$.
    \item 
    For any tree poset $T$ of height 2, $x\in T$ and $\varepsilon>0$, there exists $\delta >0$ such that for $n$ large enough any family $\cF\subseteq 2^{[n]}$ with $|\cF|\ge (1+\varepsilon)\binom{n}{\lfloor \frac{n}{2}\rfloor}$ contains a weak copy $T(x,\delta n)$.    
\end{enumerate}
\end{thm}

\section{The case of $m=2^n$}

In this section, we prove Theorem \ref{weakm=2n}, Theorem \ref{strongm=2n}, and Theorem \ref{dichotomy}. In the next subsection, we gather some lemmas about embeddings of posets, then in the second subsection we prove the above theorems.

\subsection{Embeddings of posets}

For a family $\cF\subseteq 2^{[n]}$ of sets with $\cF=\{F_1,F_2,\dots,F_{r}\}$, its $2^{r}$ \textit{atoms} are the sets $A_1,A_2,\dots,A_{2^r}$ of the form
\[
A_i=F_1^{i_1}\cap F_2^{i_2}\cap \dots \cap F_r^{i_r},
\]
where $i_j$ is either 0 or 1 and $F^1=F$ and $F^0=[n]\setminus F$. Clearly, atoms are pairwise disjoint and $F_i=\cup_{A_j\subseteq F_i}A_j$ for all $1\le i\le r$.

\begin{prop}\label{atoms}
    Suppose the atoms of $\{F_1,F_2,\dots,F_r\}$ are $A_1,A_2,\dots,A_{2^r}$. Let $B_1,B_2,\dots,B_{2^r}$ be pairwise disjoint sets with the property that for all $1\le i \le 2^r$ it holds that $A_i$ is empty if and only if $B_i$ is empty. Then the sets $G_i:=\cup_{A_j\subseteq F_i} B_j$ have the same poset structure under inclusion as the $F_i$.
\end{prop}

\begin{proof}
If $F_j\subseteq F_i$, then by definition we have $G_j=\cup_{A_h\subseteq F_j}B_h\subseteq \cup_{A_h\subseteq F_i}B_h=G_i$. 

On the other hand if $F_j\not\subseteq F_i$, then there exists a non-empty atom $A_h$ with $A_h\subseteq F_j$, $A_h\cap F_i=\emptyset$. Therefore, $B_h\subseteq G_j$, $B_j\not\subseteq G_i$ and so $B_h\subseteq G_j\not\subseteq G_i$.
\end{proof}

\begin{prop}\label{reconstruct}
   Suppose for some $j\le n$ there exists a strong copy $\cF\subseteq 2^{[n]}$ of $P$ with $F,F' \in \cF$, $F\subseteq F'$, and $|F'\setminus F|=j$. Then for any $n'>n_0(P)$ such that $n'/2-(n')^{2/3}\ge 2^{|P|}$ and $G\subseteq G'\subseteq [n']$ with $G,G'\in \cM_{n'}$ and $|G'\setminus G|\ge j$, there exists a strong copy $\cG$ of $P$ with $G,G'\in \cG$.
\end{prop}

\begin{proof}
       There are $j'\le j$ non-empty  atoms of $\cF$ contained in $F'\setminus F$ and at most $2^{|P|}$ atoms  contained both in $F$ and in $[n]\setminus F'$. As $G,G'\in \cM_{n'}$, we have $|G|,n'-|G'|\ge n'/2-(n')^{2/3}\ge 2^{|P|}$. So we can partition $G'\setminus G$, $G$, and $[n']\setminus G'$ into as many pairwise disjoint sets $B_i$ as atoms in $F'\setminus F$, $F$, and $[n]\setminus F'$. Proposition \ref{atoms} finishes the proof.
    \end{proof}

\begin{prop}\label{wcomp}
    For any poset $P$ and a pair $F\subset F' \subset [n]$ with $2\le |P|\le n$, there exists a weak copy of $P$ in $2^{[n]}$ that contains both $F$ and $F'$. 
\end{prop}

  \begin{proof}
As $C_{|P|}$ is a weak copy of $P$ for any poset $P$ and any pair of sets in containment can be extended to a maximal chain, the statement of the proposition follows.
    
    
    



\end{proof}

\begin{prop}\label{intcoint}
    Suppose $P$ is a poset with $P\neq C_k$ and let $F,F'\subseteq [n]$ be an incomparable pair with $|P|<|F\cap F'|,|F\cup F'|<n-|P|$. Then there exists a strong copy of $P$ in $2^{[n]}$ containing $F,F'$.
\end{prop}

\begin{proof}
Since  $P\neq C_k$, we can write
$P=Q\cup C_s,$
where $C_s=\{c_1\prec c_2\prec\cdots\prec c_s\}$ is a chain, possibly empty, every element of $C_s$ is smaller than every element of $Q$, and $Q$ has at least two minimal elements. Indeed, if $P$ has more than one minimal element, then we take $Q=P$ and $s=0$. Otherwise, remove the unique minimal element and put it into $C_s$. We repeat this process until the remaining poset has at least two minimal elements. Since $P$ is not a chain, this process stops before the remaining poset becomes empty.

Suppose $Q=\{q_1,\dots,q_t\}$, and set the canonical embedding $\iota_0:Q \rightarrow 2^{t}$ with $\iota_0(q_j)=\{i:q_i\preceq q_j\}$. Let $q_{1}$, $q_2$ be two distinct minimal elements of $Q$. Set $\iota_0(Q)=\bigcup_{q\in Q}\iota_0(q)$. Note that $|\iota_0(q_1)|=|\iota_0(q_2)|=1$ and $\iota_0(Q)= [t]$. Let $A, A'$ be two subsets of $[n]\backslash[t]$ such that $|A|=|F|-1$, $|A'|=|F'|-1$ and $|A\cap A'|=|F\cap F'|$. This can be done since $|F\cup F'|<n-|P|\leq n-t$.

Take a sequence of subsets $F_1\subset F_2\subset\cdots\subset F_s=A\cap A'$. This can be done since $|A\cap A'|=|F\cap F'|>|P|\geq s$. Now let $\iota: P\rightarrow 2^{[n]}$ be such that, for $c_i\in C_s$, $\iota(c_i)=F_i$, and for $q\in Q$,
$$\iota(q)=\left\{\begin{array}{ll}
    \iota_0(q)\cup(A\cap A'), & q_1\npreceq q, q_2\npreceq q; \\
    \iota_0(q)\cup A,         & q_1\preceq q, q_2\npreceq q; \\
    \iota_0(q)\cup A',        & q_1\npreceq q, q_2\preceq q; \\
    \iota_0(q)\cup(A\cup A'), & q_1\preceq q, q_2\preceq q.
\end{array}\right.$$ 
One can check that $\iota$ is a strong embedding of $P$ in $2^{[n]}$ such that $|\iota(q_1)|=|F|$, $|\iota(q_2)|=|F'|$ and $|\iota(q)\cap\iota(q')|=|F\cap F'|$. Now a permutation on $[n]$ gives a strong embedding of $P$ containing $F$ and $F'$.
\end{proof}

\begin{prop}\label{incom}
Suppose for some $i,j\le n$ there exists a strong copy $\cF\subseteq 2^{[n]}$ of $P$ with $F,F' \in \cF$, $F,F'$ incomparable, and $|F'\cap F|= i$, $n-|F'\cup F|= j$. Then for any $n'>n_0(P)$ such that $n'/2-(n')^{2/3}-|P|\ge 2^{|P|}$ and $G, G'\subseteq [n']$  incomparable with $G,G'\in \cM_{n'}$ and $|G'\cap G|\ge i$, $n'-|G'\cup G|\ge j$, there exists a strong copy $\cG$ of $P$ with $G,G'\in \cG$. 
\end{prop}

\begin{proof}
      There are $i'\le i$ non-empty atoms of $\cF$ contained in $F\cap F'$ and $j'\le j$ non-empty atoms contained in $[n]\setminus (F\cup F')$ and at most $2^{|P|}$ atoms both in $F\setminus F'$ and in $F'\setminus F$. 

      If $|G\setminus G'|,|G'\setminus G|\ge 2^{|P|}$, then one can partition $G\cap G'$, $[n']\setminus (G'\cup G)$, $G'\setminus G$, $G\setminus G'$ into as many pairwise disjoint sets as the number of atoms in $F\cap F'$, $[n]\setminus (F'\cup F)$, $F'\setminus F$, $F\setminus F'$, respectively. Proposition \ref{atoms} yields the statement.

      If at least one of them is less than $2^{|P|}$, suppose that $G\backslash G'<2^{|P|}$. 
      As $G,G'\in \cM_{n'}$, we have $|G|\geq \frac{n'}{2}-(n')^{2/3}$ and $|G'|\le n'/2+(n')^{2/3}$. Therefore, $|G\cap G'|=|G|-|G\backslash G'|>\frac{n'}{2}-(n')^{2/3}-2^{|P|}>|P|$ and $n'-|G\cup G'|=n'-|G'|-|G\setminus G'|>\frac{n'}{2}-(n')^{2/3}-2^{|P|}>|P|$ hold. By Proposition \ref{intcoint}, there exists a copy $\cG$ of $P$ with $G,G'\in \cG$.
    \end{proof}

\begin{prop}\label{goodpair}
    Suppose $P$ is a poset that contains a strong good pair $p,p'$ of incomparable elements. Then for any incomparable $F,F'\subseteq [n]$ with $2|P|<|F|,|F'|<n-2|P|$, there exists a strong copy of $P$ in $2^{[n]}$ containing $F,F'$.
\end{prop}

\begin{proof}
 By the definition of a strong good pair,  $|D_P(p)\cap D_P(p')|$ and $|U_P(p)\cap U_P(p')|$ have size at most one. If it has size one, then this unique common lower (upper) element is the smallest (largest) element of $P$.

For most part of the proof, we will assume that
$D_P(p)\cap D_P(p')=\emptyset
\text{ and }U_P(p)\cap U_P(p')=\emptyset
$. 

If $|F\setminus F'|<|P| \text{ or } |F'\setminus F|<|P|,$ we have $|F\cap F'|>|P|$ and $|F\cup F'|<n-|P|$ by $2|P|<|F|,|F'|<n-2|P|$.
Since $P$ contains two incomparable elements, it is not a chain. By Proposition \ref{intcoint}, there exists a strong copy of $P$ containing $F$ and $F'$.

Therefore, we may assume that $|F\setminus F'|\ge |P| \text{ and } |F'\setminus F|\ge |P|.$

Choose distinct elements $a,a^+\in F\setminus F',\quad b,b^+\in F'\setminus F.$

Let $L_p=D_P(p), L_{p'}=D_P(p').$ We have $L_p\cap L_{p'}=\emptyset$. Choose distinct elements $z_u\in F\setminus F'\text{ for }u\in L_p,$
and distinct elements $z_u\in F'\setminus F\text{ for }u\in L_{p'}.$
Since $|F\setminus F'|\ge |P|,|F'\setminus F|\ge |P|$ and $|L_p|,|L_{p'}|\le |P|-2$, all these elements can be distinct and distinct from $a,a^+,b,b^+$.

Let $M=P\setminus\bigl(L_p\cup L_{p'}\cup\{p,p'\}\bigr).$ For every $u\in M$, choose two elements
$\alpha_u\in [n]\setminus F, \beta_u\in [n]\setminus F'.$ Since $|[n]\setminus F|>2|P|\text{ and }|[n]\setminus F'|>2|P|,$ all chosen elements are distinct and distinct from all previously chosen markers.

Now define the mapping $\iota:P\to 2^{[n]}.$
\begin{itemize}
    \item Set $\iota(p)=F,\iota(p')=F'.$
    \item For every $u\in L_p$, let $\iota(u)=\{a\}\cup \{z_v:v\in L_p,\ v\preceq u\}.$
    \item For every $u\in L_{p'}$, let $\iota(u)=\{b\}\cup \{z_v:v\in L_{p'},\ v\preceq u\}.$
    \item For every $u\in M$, let
$$\iota(u)=\left(\bigcup_{\substack{v\in L_p\cup L_{p'}\cup \{p,p'\}\\ v\preceq u}}\iota(v)\right)
\cup
\left(\bigcup_{\substack{v\in M\\ v\preceq u}}\{\alpha_v,\beta_v\}\right).
$$
\end{itemize}

We claim that $\iota$ is a strong embedding. First, if $u\preceq v,$ then the definition of $\iota$ immediately gives $\iota(u)\subseteq \iota(v).$
Indeed, the images of elements in $L_p$ and $L_{p'}$ are defined by their downsets inside $L_p$ and $L_{p'}$, respectively. The image of an element of $M$ contains the images of all lower elements from $L_p\cup L_{p'}$, contains the private markers of all lower elements from $M$, and contains $F$ or $F'$ precisely when the element lies above $p$ or $p'$.

It remains to prove that if $u\npreceq v,$ then we have $\iota(u)\nsubseteq \iota(v).$

\textsc{Case I} $u\in L_p$. 

If $v\in L_p$ and $u\npreceq v$, then
$z_u\in\iota(u),z_u\notin\iota(v),$ so $\iota(u)\nsubseteq\iota(v).$ If $v\in L_{p'}$ or $v=p'$, then $a\in\iota(u),a\notin\iota(v).$ Since $a\in F\setminus F'$, $\iota(u)\nsubseteq\iota(v).$ If $v\in M$ and $u\npreceq v$, then the marker $z_u$ does not appear in $\iota(v)$ unless $u\preceq v$. Also, if $p\preceq v$, then $u\preceq p\preceq v$, contrary to the assumption. Therefore, $z_u\in\iota(u),z_u\notin\iota(v).$ Moreover, $\iota(u)\nsubseteq\iota(v).$ The case $u\in L_{p'}$ is similar.

\smallskip

\textsc{Case II} $u=p$. 

If $v=p'$, then $a\in F=\iota(p), a\notin F'=\iota(p'),$
so $\iota(p)\nsubseteq\iota(p').$ If $v\in L_p\cup L_{p'}$, then $a^+\in F$ does not appear in $\iota(v)$, and therefore $\iota(p)\nsubseteq\iota(v).$ If $v\in M$ and $p\npreceq v$, then  $\iota(v)$ does not include $F$. Moreover, $a^+\notin F'$. Hence, $a^+\in\iota(p),a^+\notin\iota(v),$ which gives $\iota(p)\nsubseteq\iota(v).$
The case $u=p'$ is similar.

\smallskip

\textsc{Case III} $u\in M$. 

If $v=p$, then $\alpha_u\in\iota(u),\alpha_u\notin F=\iota(p),$ so $\iota(u)\nsubseteq\iota(p).$ If $v=p'$, then
$\beta_u\in\iota(u),\beta_u\notin F'=\iota(p'),$ so $\iota(u)\nsubseteq\iota(p').$
If $v\in L_p\cup L_{p'}$, then neither $\alpha_u$ nor $\beta_u$ appears in $\iota(v)$, and therefore
$\iota(u)\nsubseteq\iota(v).$ Finally, if $v\in M$ and $u\npreceq v$. If $p\preceq v$, then $p'\npreceq v$, because no element of $M$ is above both $p$ and $p'$. In this case, the element
$\alpha_u\in [n]\setminus F$ does not appear in $\iota(v)$ unless $u\preceq v$. Hence $\iota(u)\nsubseteq\iota(v).$ If $p'\preceq v$, then the same argument with $\beta_u\in [n]\setminus F'$ gives $\iota(u)\nsubseteq\iota(v).$ If neither $p\preceq v$ nor $p'\preceq v$, then $\alpha_u$ does not appear in $\iota(v)$ unless $u\preceq v$. Thus $\iota(u)\nsubseteq\iota(v).$

\smallskip

Therefore, for all elements $u,v$, we have
$u\preceq v\Longleftrightarrow\iota(u)\subseteq\iota(v)$. Moreover, for any $u\in P$, $\emptyset \neq \iota(u)\neq [n]$ holds. Indeed, the image of $\iota$ is nonempty and if $\iota(u_0)=[n]$, then $\iota(p)\subset\iota(u_0)$ and $\iota(p')\subset\iota(u_0)$ and $u_0\in U_P(p)\cap U_P(p')$ hold, which contradicts the assumption. 

The proof above applies if  $D_P(p)\cap D_P(p')=U_P(p)\cap U_P(p')= \emptyset$.
If $D_P(p)\cap D_P(p')$ or $U_P(p)\cap U_P(p')\neq \emptyset$, then we apply the above proof to $P_0=P\setminus\bigl((D_P(p)\cap D_P(p'))
    \cup(U_P(p)\cap U_P(p'))\bigr)$, and then we set the image of the at most one element in $D_P(p)\cap D_P(p')$ to be $\emptyset$ and the image of the at most one element in $U_P(p)\cap U_P(p')$ to be $[n]$.
We obtain a strong embedding of $P$ into $2^{[n]}$ whose image is a strong copy of $P$ containing both $F$ and $F'$.
\end{proof}

\subsection{Proofs of Theorems \ref{weakm=2n}, \ref{strongm=2n}, \ref{dichotomy}}

We start with an important consequence of Proposition \ref{intcoint}. Recall that $\cM_n=\{F\subseteq [n]: ||F|-n/2|\le n^{2/3}\}$.

\begin{prop}\label{independent}
    For any $P\neq C_k$, there exists $n_0=n_0(P)$ such that if $n\ge n_0$, then an independent set of $G^*_P[\cM_n]$ (and therefore that of $G_P[\cM_n]$) cannot contain an antichain of size 3. In particular, by Dilworth's theorem, an independent set is always the union of at most 2 chains.
\end{prop}

\begin{proof}
    If the incomparable $F,F'\in \cM_n$ are not contained in a strong copy of $P$, then by Proposition \ref{intcoint}, $|F\cap F'|\le |P|$ or $n-|F\cup F'|\le |P|$ happens. As $F,F'\in \cM_n$, both conditions imply $|F\cap F'|\le |P|+2n^{2/3}$ and $n-|F\cup F'|\le |P|+2n^{2/3}$. 
    Suppose that an independent set $\mathcal{I} \subseteq \cM_n$ contains an antichain of size 3, saying $F, F', F''\in \mathcal{I}.$ We have $|F\cap F''|,|F'\cap F''|\le |P|+2n^{2/3}$ and $n-|F\cup F'|\le |P|+2n^{2/3}$. Every element of $F''$ either belongs to $F$, belongs to $F'$, or lies outside $F \cup F'$. Hence, $$|F''|\leq |F\cap F''|+|F'\cap F''|+n-|F\cup F'|\leq 3|P|+6n^{2/3}.$$ 
    On the other hand, since $F''\in \cM_n$, $|F''|\geq n/2-n^{2/3}$. For all sufficiently large $n$, $n/2-n^{2/3}>3|P|+6n^{2/3}$ which gives a contradiction.

\end{proof}
    
\begin{proof}[Proof of Theorem \ref{weakm=2n}]
If $P$ has a weak good incomparable pair, the upper bound is trivial. If all incomparable pairs are weak bad, then no weak copy of $P$ can contain both $F$ and $\overline F$ for any non-empty $F$.  Indeed, if there exists a copy containing $F,\overline F$, then because $F\cap\overline F=\emptyset$ and $F\cup\overline F=[n]$, there could be at most one common upper element and at most one common lower element, which  contradicts the definition of weak bad.  Thus each complementary pair may be given one color, with $\emptyset$ and $[n]$ colored separately, and so $\arw(n,2^n,P)\le \frac{1}{2}\cdot2^n+1$. This finishes the proofs of the upper bounds.

For the lower bounds, let us consider $G_P[\cM_n]$. By Proposition \ref{wcomp}, for any large enough $n$ and $F,F'\in\cM_n$ with $F\subset F'$ there exists a weak copy of $P$ containing $F,F'$, so a chain in an independent set of $G_P[\cM_n]$ can contain at most one set. By Proposition \ref{independent}, an independent set is the union of at most 2 chains. Hence $\alpha(G_P[\cM_n])\le 2$, and so 
$$
\operatorname{ar}_{\mathrm{m}}\left(n, 2^n, P\right)=\chi\left(G_P\left(2^{[n]}\right)\right) \geq \chi\left(G_P\left[\cM_n\right]\right) \geq \frac{\left|\cM_n\right|}{\alpha\left(G_P\left[\cM_n\right]\right)} \geq \frac{\left|\cM_n\right|}{2}=\left(\frac{1}{2}-o(1)\right) 2^n .
$$

This lower bound together with the upper bound establishes the asymptotics of $\arw(n,2^n,P)$ if $P$ is weak incomparable bad. When $P$ has a weak good incomparable pair $p,p'$, then if $p,p'$ is strong good incomparable pair by Proposition \ref{wcomp} and Proposition \ref{goodpair}, we have $G_P[\cM_n]$ is a clique. If $|U_P(p)\cap U_P(p')|=1$ and $P$ does not have a largest element, we can let $\iota(U_P(p)\cap U_P(p'))=[n]$ and if $|D_P(p)\cap D_P(p')|=1$ and $P$ does not have a smallest element, we can let $\iota(D_P(p)\cap D_P(p'))=\emptyset$. 
Similar to the proof of Proposition \ref{goodpair}, we can obtain that $\iota$ is a weak embedding. Hence we have $G_P[\cM_n]$ is a clique and so
$$\operatorname{ar}_{\mathrm{m}}\left(n, 2^n, P\right)=\chi\left(G_P\left(2^{[n]}\right)\right) \geq \chi\left(G_P\left[\cM_n\right]\right) \geq\left|\cM_n\right|=(1-o(1)) 2^n.$$
\end{proof}


As both proofs proceed by a case analysis whether $P$ is incomparable good or bad and if bad then lower, upper or mixed, we will prove Theorem \ref{strongm=2n} and Theorem \ref{dichotomy} simultaneously. Our constructions for covering $G_P^*[\cM_n]$ will use chain decompositions of $2^{[n]}$. A family $\mathbb{C}_n=\{\cC_1,\cC_2,\dots,\cC_m\}$ of pairwise disjoint chains with $\bigcup_{i=1}^m\cC_i=2^{[n]}$ is a \textit{symmetric chain decomposition} if for all $i$ there exists $r$ such that $\cC_i=\{C_r\subset C_{r+1}\subset \dots \subset C_{n-r}\}$ with $|C_j|=j$ for all $r\le j \le n-r$. It is well-known that symmetric chain decompositions of $2^{[n]}$ exist for all $n\ge 1$. Clearly, $m=\binom{n}{\lfloor n/2\rfloor}=o(2^n)$.

Given a symmetric chain decomposition $\mathbb{C}_n$ of $2^{[n]}$, we obtain a \textit{twisted chain decomposition} $\mathbb{C}_n^\perp$ of $2^{[n]}$ as follows: for a chain $\cC\in \mathbb{C}_n$ we write $\cC^-=\{C\in \cC: |C|\le n/2\}$ and set $\cC^\perp
=\{[n]\setminus C: C\in \cC, |C|<n/2\}$ (note that if $n$ is even and $\cC$ contains only one element, then $\cC^\perp$ is empty). Then $\mathbb{C}_n^\perp:=\{\cC^-,\cC^\perp:\cC\in \mathbb{C}_n\}$ is a chain decomposition of $2^{[n]}$. 

\begin{proof}[Proof of Theorem \ref{strongm=2n} and Theorem \ref{dichotomy}]
    The lower bound of Theorem \ref{strongm=2n} follows immediately from Proposition \ref{chromatic} and the induced subgraph on $\cM_n$:
$$
    \arm(n,2^n,P)
    =\chi(G_P^*(n))
    \ge \chi(G_P^*[\cM_n])
    \ge \frac{|\cM_n|}{\alpha(G_P^*[\cM_n])}
    =\left(\frac{1}{\alpha(G_P^*[\cM_n])}-o(1)\right)2^n.
$$

\textsc{Case I} $P$ is strong incomparable good.

\smallskip

Then by Proposition \ref{goodpair}, any incomparable $F,F'\in \cM_n$ are contained in a strong copy of $P$. Therefore, any independent set of $G_P^*[\cM_n]$ is a chain. By definition of $\sigma(P)$, any chain $F_0\subset F_1\subset \dots \subset F_{\sigma(P)-1}$ with $|F_j|=|F_0|+j$ forms an independent set in $G^*_P[\cM_n]$ so $\alpha(G^*_P[\cM_n])\ge \sigma(P)$. On the other hand, suppose $|F'\setminus F|=\sigma(P)$ with $F\subseteq F'$ such that there exists a strong copy of $P$ containing $F,F'$. Then by Proposition \ref{reconstruct}, if $n$ is large enough, for all $G,G'\in \cM_n$ with $G\subseteq G'$ and $|G'\setminus G|\ge \sigma(P)$, $G,G'$ are contained in a strong copy of $P$, so chains corresponding to independent sets
of $G^*_P[\cM_n]$ cannot have length more than $\sigma(P)$. This shows $\alpha(G_P^*[\cM_n])=\sigma(P)$. To color $G^*_P[\cM_n]$ consider a symmetric chain partition of $\mathbb{C}_n$ of $2^{[n]}$. We partition each chain into blocks of $\sigma(P)$ consecutive sets. By the above, these form independent sets, so we can color them using a color assigned to this block. On each chain, less than $\sigma(P)$ sets are uncovered by blocks: we color these sets one-by-one. So the total number of colors used is at most $\frac{1}{\sigma(P)}2^n+(\sigma(P)-1)\binom{n}{\lfloor n/2\rfloor}=(\frac{1}{\sigma(P)}+o(1))2^n$.

\medskip

\textsc{Case II} $P$ is incomparable upper bad or incomparable lower bad.

\smallskip

Then any $F,\bar{F}$ with $\{F,\bar{F}\}\neq \{\emptyset,[n]\}$ is a non-edge in $G_P^*(n)$ and using these pairs as color classes, we obtain $\chi(G_P^*(n))\le \frac{1}{2}2^n+1$ and $\alpha(G_P^*[\cM_n])\ge 2$. It remains to show $\alpha(G_P^*[\cM_n])\le 2$ which will follow from $\sigma(P)=1$ and Proposition \ref{independent}. Suppose $P$ is upper incomparable bad. Then $P$ contains a largest element, as if not, then two maximal elements $p,q$ are incomparable, therefore strong upper bad, so $U_P(p)\cap U_P(q)\neq \emptyset$ which contradicts the maximality of $p$ and $q$. Let $1$ be the largest element of $P$. Then $P'=P\setminus \{1\}$ should also have a largest element as otherwise $p,q$ maximal elements of $P'$ would be strong upper bad in $P$. But $P$ does have a largest element, so $|U_P(p)\cap U_P(q)|\ge 2$ contradicts the maximality of $p,q$ in $P'$. Let $1'\in P'$ be the largest element and let $P''=P\setminus\{1,1'\}$. Then for any $F'\subseteq F$ with $|F|=|F'|+1$, $|F'|\ge |P|-1$, we can define an embedding $\iota$ by picking distinct elements $x_1,x_2,\dots,x_{|P|-2}\in F'$
and for $p_1,p_2,\dots, p_{|P|-2}\in P''$ and setting $\iota(p_i)=\{x_j: p_j\preceq p_i\}$, $\iota(1')=F'$, and $\iota(1)=F$. This shows $\sigma(P)=1$ as claimed. The case of lower incomparable bad $P$ is analogous.

\medskip

\textsc{Case III} $P$ is mixed incomparable bad.

\smallskip

If a largest independent set in $G^*_P[\cM_n]$ is a chain, then $\alpha(G^*_P[\cM_n])\le \sigma(P)$ as in Case I, and all chains $A_0\subset A_1\subset \dots \subset A_r$ with $|A_r\setminus A_0|<\sigma(P)$ form an independent set, so $\alpha(G^*_P[\cM_n])= \sigma(P)$ and using the symmetric chain decomposition as in Case I, we obtain $\arm(n,2^n,P)=(\frac{1}{\sigma(P)}+o(1))2^n$.

So, by Proposition \ref{independent}, we can assume that a largest independent set in $G^*_P[\cM_n]$ is a union of two chains $A_0\subset A_1\subset \dots A_{r-1}$ and $B_0\subset B_1 \subset \dots B_{s-1}$. We claim that all pairs $A_i,B_j$ are incomparable. Suppose not and $A_i\subset B_j$ for some $i$ and $j$. There exist $i',j'$ such that $A_{i'},B_{j'}$ are incomparable, as otherwise the two chains would form one single chain. But then, $|A_{i'}\cup B_{j'}|\le |B_j|+2n^{2/3}\le n-|P|$ and $|A_{i'}\cap B_{j'}|\ge |A_i|-2n^{2/3}\ge |P|$ if $n$ is large enough, and then by Proposition \ref{intcoint}, $A_{i'},B_{j'}$ would form an edge in $G_P^*[\cM_n]$. This contradiction shows that all $A_i,B_j$ are incomparable.

Next we claim that we can assume that $A_0\cap B_0=\emptyset$ and $A_{r-1}\cup B_{s-1}=[n]$. Suppose first $A_0\cap B_0\neq \emptyset$ with $|A_0|\ge |B_0|$. If $|A_0|-|A_0\cap B_0|\ge n/2-n^{2/3}$, then, writing $A'_i=A_i\setminus (A_0\cap B_0)\in \cM_n$, we claim that the $A_i'$ and the $B_j$ form an independent set of the same size $r+s$. The $A'_i$ induce an independent set because of Proposition \ref{reconstruct} and the fact $|A'_i\setminus A'_h|=|A_i\setminus A_h|$, while $A'_i,B_j$ is not an edge because of Proposition \ref{incom} and the facts $|A'_i\cap B_j|\le |A_i\cap B_j|$, $|A'_i\cup B_j|=|A_i\cup B_j|$. If $|A_0|-|A_0\cap B_0|< n-n^{2/3}$, then as $|A_{r-1}\setminus A_0|,|B_{s-1}\setminus B_0|\le \sigma(P)$, there exists a set $X\subseteq [n]\setminus (A_{r-1}\cup B_{s-1})$ of size $|A_0\cap B_0|\le \lambda^-(P)$. So writing $A''_i=(A_i\cup X)\setminus (A_0\cap B_0)$, the $A_i''$ and the $B_j$ form an independent set of size $r+s$. The $A''_i$ induce an independent set because of Proposition \ref{reconstruct} and the fact $|A''_i\setminus A''_h|=|A_i\setminus A_h|$, while $A''_i,B_j$ is not an edge because of Proposition \ref{incom} and the facts $|A''_i\cap B_j|\le |A_i\cap B_j|$, $|A''_i\cup B_j|\ge |A_i\cup B_j|$.
The proof for $A_{r-1}\cup B_{s-1}\neq [n]$ is analogous.
 
Let us bound $r+s$. As the $A_i$s and the $B_j$s form a chain, we have $r,s\le \sigma(P)$ and thus $r+s\le 2\sigma(P)$. Also, by definition and the assumptions $A_0\cap B_0=\emptyset$, $A_{r-1}\cup B_{s-1}=[n]$, we have $|[n]\setminus (A_0\cup B_0)|\le \lambda^+(P)-1$ and $|A_{r-1}\cap B_{s-1}|\le \lambda^-(P)-1$. For $0\le i\le r-1,0\le j\le s-1$, set $\kappa(i,j):=(n-|A_i\cup B_j|,|A_i\cap B_j|)$. Consider the sequence, $\kappa(0,0),\kappa(0,1),\dots,\kappa(0,s-1),\kappa(1,s-1),\dots,\kappa(r-1,s-1)$. The sequence starts with $(a,0)$ for some $a\le \lambda^+(P)-1$ and ends with $(0,b)$ for some $b\le \lambda^-(P)-1$. Also, as we always change a set to its superset, the first coordinate can never increase, the second coordinate can never decrease. Consecutive pairs, and therefore all pairs are different as we either introduce a new element to the intersection or to the union of the two sets. Therefore, the length $r+s-1$ of the sequence satisfies $r+s-1\le a+b+1\le \lambda^+(P)+\lambda^-(P)-1$. We obtained $\alpha(G_P^*[\cM_n])\le \min\{2\sigma(P),\lambda^+(P)+\lambda^-(P)\}$ as claimed.

To prove $\alpha(G_P^*[\cM_n])= \min\{2\sigma(P),\lambda^+(P)+\lambda^-(P)\}$, we need a construction. Let $r=\lfloor \frac{\lambda^+(P)+\lambda^-(P)}{2}\rfloor$ and $s=\lceil \frac{\lambda^+(P)+\lambda^-(P)}{2}\rceil$. Consider a chain pair $\cC^-,\cC^\perp$ from a twisted chain partition with $|\cC^\perp |\ge s$. Let $C_0\in \cC^-$ with $|C_0|<n/2-r$. Then there exists a  set $C_0^\perp\in \cC^\perp$ of size $n-|C_0|-\lambda^+(P)+1$. Consider the $r$ consecutive sets $C_0\subset C_1\subset \dots \subset C_{r-1}$ of $\cC^-$ with $|C_i|=|C_0|+i$ and the $s$ consecutive sets $C_0^\perp \subset C_1^\perp \subset \dots \subset C_{s-1}^\perp$ of $\cC^\perp$ with $|C_j^\perp|=|C_0^\perp|+j$. Observe that $|C_{r-1}\cap C_{s-1}^\perp|=r-1+s-1-(n-|C_0\cup C_0^\perp|)=r+s-1-\lambda^+(P)=\lambda^-(P)-1$. Also, by the definition of twisted chain partition, for all sets $C\in \cC^-, C^\perp \in \cC^\perp$ we either have $C\cap C^\perp=\emptyset$ or $C\cup C^\perp=[n]$. So for any $i,j$ we have that $(|C_i\cap C_j^\perp|,n-|C_i\cup C_j^\perp|)$ is either $(0,a)$ for some $a\le \lambda^+(P)-1$ or $(b,0)$ for some $b\le \lambda^-(P)-1$. So by Proposition \ref{incom}, $C_i,C_j^\perp$ do not form an edge in $G_P^*[\cM_n]$. So the only obstacle that can prevent the $C_i$s and $C_j^\perp$s to form an independent set is an edge within the $C_i$s or $C^\perp_j$s, which, by Proposition \ref{reconstruct}, happens if $s>\sigma(P)$. In this case we can take $\sigma(P)$ consecutive sets from the $C_i$ and from the $C_j^\perp$s to form an independent set of size $2\sigma(P)$. If $s\le \sigma(P)$, then the $C_i$s and the $C_j^\perp$s form an independent set of size $r+s=\lambda^+(P)+\lambda^-(P)$. This proves $\alpha(G_P^*[\cM_n])=\min\{2\sigma(P),\lambda^+(P)+\lambda^-(P)\}$.
 
Finally, we need a proper coloring of $G_P^*[\cM_n]$ with $(\frac{1}{\alpha(G_P^*[\cM_n])}+o(1))|\cM_n|=(\frac{1}{\alpha(G_P^*[\cM_n])}+o(1))2^n$ colors as then coloring $2^{[n]}\setminus \cM_n$ set-by-set with new colors yields a coloring of $G_P^*(n)$ showing $$\chi(G_P^*(n))\le \left(\frac{1}{\alpha(G_P^*[\cM_n])}+o(1)\right)|\cM_n|=\left(\frac{1}{\alpha(G_P^*[\cM_n])}+o(1)\right)2^n.$$ 
To this end, observe that the role of $r$ and $s$ are interchangeable in the
above construction, so one could have taken $s$ sets from $\cC^-$ and $r$ sets from $\cC^\perp$. Therefore, any pair $\cC^-,\cC^\perp$ in a twisted chain partition can be partitioned, with the exception of at most $\lambda^+(P)+\lambda^-(P)$ sets, into consecutive blocks $(B_1,B^\perp_1),(B_2,B_2^\perp),\dots$ such that all $B_i\cup B_i^\perp$ form an independent set with $|B_{2j+1}|=r,|B_{2j+1}^\perp|=s,|B_{2j}|=s,|B_{2j}^\perp|=r$ if $s\le \sigma(P)$ and $|B_i|=|B_i^\perp|=\sigma(P)$ if $s>\sigma(P)$. As the number of chain pairs in a twisted chain partition is $\binom{n}{\lfloor n/2\rfloor}=o(2^n)$, the number of sets uncovered by these block is $o(2^n)$ and sets in $B_i\cup B_i^\perp$ can be colored using the same color. As each of these color classes have size $\alpha(G_P^*[\cM_n])$, this finishes the proof of the construction.

Together with the lower bound, this proves the theorem.
\end{proof}

We have not yet considered possible values of the parameters $\sigma(P)$, $\lambda^+(P)$, $\lambda^-(P)$.

Let $S_k$ be the \textit{standard example} poset with elements $a_1,\dots,a_k,b_1,\dots,b_k$ with all relations being $a_i\prec b_j$ whenever $i\neq j$. The pairs $(a_i,b_i)$ are incomparable good, so $S_k$ is incomparable good. We claim $\sigma(S_k)=k-2$ for $k\geq 3$. The embedding $a_i\mapsto \{i\}, b_i\mapsto [k]\setminus \{i\}$ shows $\sigma(S_k)\le (k-1)-1=k-2$. To see the lower bound, observe that $\iota(a_i)\setminus \cup_{j:j\neq i}\iota(a_j)\neq \emptyset$ must hold for all $i$ and embedding $\iota$ as otherwise $\iota(a_i)\subseteq \cup_{j:j\neq i}\iota(a_j) \subseteq \iota(b_i)$ would contradict $\iota$ being an embedding. Therefore, denoting by $x_h$ an element of  $\iota(a_h)\setminus \cup_{j:j\neq h}\iota(a_j)$, for any $i\neq j$ we have $\iota(b_j)\setminus \iota(a_i)\supset \{x_h: h\neq i,j\}$ and so $|\iota(b_j)\setminus \iota(a_i)|\ge k-2$.

One can define $S^\ell_k$ on elements $a^j_i$ with $1\le j\le \ell+1$, $1\le i\le k$ with cover relations $a_i^j\prec a_{i'}^{j+1}$ for all $i\neq i'$ and $a_i^j\prec a_i^{j+2}$. Observe $S_k=S^1_k$. For $\ell\ge 2$, $S^\ell_k$ is strong incomparable bad, but it is not hard to see $2(k-2)=2\sigma(S^\ell_k)\le \lambda^+(S^\ell_k)+\lambda^-(S^\ell_k)$. We have not identified any strong incomparable  posets $P$ with $\lambda^+(P)+\lambda^-(P)<2\sigma(P)$ or with $\lambda^+(P),\lambda^-(P)$ having different parities.

\section{The case of $m=\Theta(\binom{n}{\lfloor n/2\rfloor})$}

Let us recall that $\cM_{n,k}=\cup_{i=1}^{k}\binom{[n]}{\lfloor \frac{n-k}{2}\rfloor+i}$ denotes the family of the $k$ middle layers. For $n>k$ and $L\subseteq \{0,1,\dots,k-1\}$, the generalized Johnson graph $J(n,k,L)$ has vertex set $\binom{[n]}{k}$ and $F,G$ are connected by an edge if $|F\cap G|\in L$. For a recent survey on Johnson graphs, their independence, clique, and chromatic number see \cite{Rai}.

For any two pairs, $F,G$ and $F',G'$ with $|F|=|F'|$, $|G|=|G'|$ and $|F\cap G|=|F'\cap G'|$, there is an automorphism $\phi$ of $(2^{[n]},\subseteq)$ that sends $F$ to $F'$, $G$ to $G'$, and $F\cap G$ to $F'\cap G'$. So if a strong (weak) copy $\cH\subseteq \cM_{n,k}$ of a poset $P$ contains $F$ and $G$, then the strong (weak) copy $\phi(\cH)$ contains $F',G'$.
Therefore the fact whether $F$ and $G$ are connected in $G_P(\cM_{n,k})[\binom{[n]}{\lfloor \frac{n-k}{2}\rfloor +i}]$ or in $G^*_P(\cM_{n,k})[\binom{[n]}{\lfloor \frac{n-k}{2}\rfloor +i}]$ depends only on the size of their intersection. So $G_P(\cM_{n,k})[\binom{[n]}{\lfloor \frac{n-k}{2}\rfloor +i}]$ is isomorphic to $J(n,\lfloor \frac{n-k}{2}\rfloor +i,L_{P,n,k,i})$ and $G^*_P(\cM_{n,k})[\binom{[n]}{\lfloor \frac{n-k}{2}\rfloor +i}]$ is isomorphic to $J(n,\lfloor \frac{n-k}{2}\rfloor +i,L^*_{P,n,k,i})$, where $L_{P,n,k,i}$ and $L^*_{P,n,k,i}$ are subset of $\{0,1,\dots,\lfloor \frac{n-k}{2}\rfloor +i-1\}$.

\begin{proof}[Proof of Proposition \ref{exponent}]
    Let us define $E_{P,n,k,i}=\{\lfloor \frac{n-k}{2}\rfloor +i-\ell: \ell \in L_{P,n,k,i}\}$ and define $m_{P,k,i}:=\max\bigl(\{0\}\cup E_{P,n,k,i}\bigr)$. For fixed $P, k, i$, the sets $E_{P, n, k, i}$ are eventually constant as $n \rightarrow \infty$. Indeed, every realization can be extended to a larger Boolean lattice by adding each new ground element either to all sets or to none of them, according to the shift of the middle layers. As $P$ is connected, if the diameter of the Hasse diagram of $P$ is $d$, then $m_{P,k,i}\le d(k-1)$. Indeed, if $F,G$ are sets of size $\lfloor \frac{n-k}{2}\rfloor +i$ contained in a copy of $P$, then there is a path of length at most $d$ in $H(P)$ connecting the elements $p,q\in P$ that they play the role of. At each edge of $H(P)$ we might remove or add at most $k-1$ vertices from $F$ to get closer to $G$. 

    By a coding theory construction of Graham and Sloane \cite{GS}, we have $\chi(J(n,\lfloor \frac{n-k}{2}\rfloor +i,L_{P,n,k,i}))\le \chi(J(n,\lfloor \frac{n-k}{2}\rfloor +i,[\lfloor\frac{n-k}{2}\rfloor+i-m_{P,k,i},\lfloor\frac{n-k}{2}\rfloor+i-1]))=O(n^{m_{P,k,i}})$ and thus letting $m_{P,k}:=\max \{m_{P,k,i}:0< i\le k\}$, by coloring layers $\binom{[n]}{\lfloor \frac{n-k}{2}\rfloor +i}$ separately, using distinct color sets, we obtain that   $$\chi(G_P(\cM_{n,k}))\le \sum_{i=1}^{k}\chi\left(G_P(\cM_{n,k})\left[\binom{[n]}{\lfloor \frac{n-k}{2}\rfloor +i}\right]\right)\le \sum_{i=1}^{k}O(n^{m_{P,k,i}})=O( n^{m_{P,k}})$$ for fixed $k$ and $P$ while $n$ grows to infinity.

It is not hard to see that $\chi(G_P(\cM_{n,k}))=\Theta(n^{m_{P,k}})$ holds. Indeed, if $m_{P,k}=0$, then $\chi(G_P(\cM_{n,k}))\ge 1=\Theta(n^0)$. If $m_{P,k}\ge 1$, then we use a result of Frankl and F\"uredi \cite{FF} that states if $t,\ell$ are fixed integers and $n\ge n_0(t,\ell)$, then any $\cF\subseteq \binom{[n]}{\ell}$ with $|F\cap F'|\neq t$ has size $O(n^{\max \{t,\ell-t-1\}})$. Suppose $m_{P,k}=m_{P,k,i}$. Then an independent set in $G_P(\cM_{n,k})[\binom{[n]}{\lfloor \frac{n-k}{2}\rfloor+i}]$ corresponds to a family $\cF$ for which $|F\cap F'|\notin L_{P,n,k,i}$. So we can fix an integer $\ell$ with $\ell>2m_{P,k}$ and count pairs $(B,F)$ with $B\subset F\in \cF$ and $|B|=\lfloor \frac{n-k}{2}\rfloor+i-\ell$. For every $B\in \binom{[n]}{\lfloor \frac{n-k}{2}\rfloor+i-\ell}$, define $\mathcal{F}_B=\{F\setminus B:F\in \mathcal{F},\ B\subseteq F\}$. Then $\cF_B$ is an $\ell$-uniform family. Moreover, if $X,Y\in \cF_B$ are distinct and $|X\cap Y|=\ell-m_{P,k}$, then the corresponding sets $F=B\cup X$ and $F'=B\cup Y$ satisfy $|F\cap F'|=|B|+|X\cap Y|=\lfloor \frac{n-k}{2}\rfloor+i-m_{P,k}\in L_{P,n,k,i}$, which is a contradiction. Hence, no two of $\cF_B$ intersect in exactly $\ell-m_{P,k}$ elements. The number of such pairs is exactly $|\cF|\cdot \binom{\lfloor \frac{n-k}{2}\rfloor+i}{\ell}$, but also, by the Frankl-F\"uredi result, at most $\binom{n}{\lfloor \frac{n-k}{2}\rfloor+i-\ell}\cdot O(n^{\ell-m_{P,k}})$. We obtain \[
|\cF|\le O\left(\frac{\binom{n}{\lfloor \frac{n-k}{2}\rfloor+i-\ell}}{\binom{\lfloor \frac{n-k}{2}\rfloor+i}{\ell}}\frac{n^\ell}{n^{m_{P,k}}}\right)=O\left(\frac{1}{n^{m_{P,k}}}\binom{n}{\lfloor \frac{n-k}{2}\rfloor+i}\right),
\]
and thus $$\chi(G_P(\cM_{n,k}))\ge \chi\left(G_P(\cM_{n,k})\left[\binom{[n]}{\lfloor \frac{n-k}{2}\rfloor+i}\right]\right) \ge \frac{\binom{n}{\lfloor \frac{n-k}{2}\rfloor+i}}{\alpha(G_P(\cM_{n,k})[\binom{[n]}{\lfloor \frac{n-k}{2}\rfloor+i}])}=\Omega(n^{m_{P,k}}).$$

The proof of the statement for $\chi(G^*_P(\cM_{n,k}))$ is analogous.
\end{proof}

\begin{rem}
    Note that the statement of Proposition \ref{exponent} does not remain true for disconnected posets. Indeed, let $P=P_1\cup P_2$ with $p_1 \perp p_2$ for all $p_1\in P_1,p_2\in P_2$. Suppose $P$ and thus $P_1$ and $P_2$ embed into $\cM_{n,k}$ and let $G_1,G_2\in \cM_{n+2,k}$ with $|G_1|=|G_2|$ and $x\in G_1\setminus G_2$, $y\in G_2\setminus G_1$. Then one can embed $P_1$ and $P_2$ into the $k$ middle layers of $2^{[n]\setminus \{x,y\}}$ such that the copy $\cF_1$ of $P_1$ contains $G_1\setminus \{x\}$ and the copy $\cF_2$ of  $P_2$ contains $G_2\setminus \{y\}$. Then $\{G\cup \{x\}:G\in \cG_1\}\cup \{G\cup \{y\}:G\in \cG_2\}$ is a copy of $P$ containing $G_1,G_2$. Thus we obtain that $G^*_P[\binom{[n+2]}{\frac{n+2-k}{2}+i}]$ is a clique.
\end{rem}

Let us state one part of the previous proof as a separate corollary. Analogously to $E_{P,n,k,i}$, we define $E^*_{P,n,k,i}=\{\lfloor \frac{n-k}{2}\rfloor +i-\ell: \ell \in L^*_{P,n,k,i}\}$.

\begin{cor}\label{E}
    For any poset $P$ and integer $k$, we have 
   \[
m_{P,k}
=\max\left(\{0\}\cup\bigcup_{i=1}^k E_{P,n,k,i}\right),
\qquad
m^*_{P,k}
=\max\left(\{0\}\cup\bigcup_{i=1}^k E^*_{P,n,k,i}\right).
\]
\end{cor}

For tree posets $T$, we know more. From Corollary \ref{E} and (2) of the next Observation, we have $m^*_{T,k}=m_{T,k}$ for all integers $k$ and tree posets $T$. 
We say that $r:T\rightarrow [1,k]$ is a \textit{rank function of $T$} if $p\prec_T q$ implies $r(p)<r(q)$. Also, let us define the canonical rank function $p:T\rightarrow [1,h(T)]$ as follows: let $T=\cup_{i=1}^{h(T)}\cM_i$, where $\cM_{j+1}$ is the set of minimal elements of $T\setminus \cup_{i=1}^j\cM_i$. Then $f(x)=i$ if and only if $x \in \cM_i$.

Let us define $L_{T,n,k}(i)$ as follows: we set $j\in L_{T,n,k}(i)$ if and only if there exist a rank function $r:T\rightarrow [1,k]$  and elements $x,y\in T$ with $r(x)=r(y)=i$ such that $j=\lfloor \frac{n-k}{2}\rfloor +i-\frac{1}{2}\sum_{h=1}^d|r(x_h)-r(x_{h-1})|$, where $x=x_0,x_1,\dots,x_d=y$ is the only path between $x$ and $y$ in the Hasse diagram of $T$.

\begin{obs}\label{layer}
    If $n$ is large enough with respect to $k\ge h(T)$ and $|T|$, then
    \begin{enumerate}
        \item        
        $L_{T,n,k}(i)\subseteq L^*_{T,n,k,i}$, i.e. a pair $F,G\in \binom{[n]}{\lfloor \frac{n-k}{2}\rfloor +i}$ with $|F\cap G|\in L_{T,n,k}(i)$ forms an edge in $G_T(\cM)$ and in $G^*_T(\cM)$,
        \item 
        $\min L_{T,n,k,i}=\min L_{T,n,k}(i)$.
    \end{enumerate}
\end{obs}

\begin{proof}
To show (1), let $r$ be the rank function showing $|F\cap G|\in L_{T,n,k}(i)$ and let $x$ and $y$ be the two elements of $T$ with $|F\cap G|=\lfloor\frac{n-k}{2}\rfloor+i-\frac{1}{2}\sum_{h=1}^d|r(x_h)-r(x_{h-1})|$. We shall define a strong copy $\cT$ of $T$ containing both $F$ and $G$. For an edge $e$ of the Hasse diagram of $T$, we will say that its \textit{color set} is $H\setminus H'$, where $H'\subset H$ play the role of the endpoints of $e$ in $\cT$. First we define the sets of a copy $\cT$ for the path $x=x_0,x_1,\dots,x_d=y$. We set $F$ to play the role of $x$ and $G$ to play the role of $y$. For any $h$ with $r(x_h)>r(x_{h-1})$ if $F_{h-1}$ is defined, then we set $F_h$ by adding $r(x_h)-r(x_{h-1})$ elements of $G\setminus F$ disjoint with $F_{h-1}$, while if $r(x_h)<r(x_{h-1})$, then we remove $r(x_{h-1})-r(x_{h})$ elements of $(F\setminus G)\cap F_{h-1}$ from $F_{h-1}$ to obtain $F_h$. As $|F\cap G|=\lfloor\frac{n-k}{2}\rfloor+i-\frac{1}{2}\sum_{h=1}^d|r(x_h)-r(x_{h-1})|$ and $|F|=|G|=\lfloor\frac{n-k}{2}\rfloor+i$, all elements of $F\setminus G$ are removed and all elements of $G\setminus F$ are added. To finish the definition of $\cT$, we extend our embedding by edges of the Hasse diagram: we always add or remove $|r(t)-r(t')|$ elements that are not contained in any previous color sets. This is possible as $n$ is large enough and all sets are of size roughly $n/2$ and $T$ is fixed.

To show (2), let $\cT$ be a weak copy of $T$ in $\cM_{n,k}$ containing $F,G$ of size $\lfloor \frac{n-k}{2}\rfloor+i$ corresponding to $x,y\in T$. For $t\in T$, let us define $r(t)=|H_t|-\lfloor \frac{n-k}{2}\rfloor$, where $H_t$ is the set corresponding to $t$. Now clearly, $|F\triangle G|\le \sum_{h=1}^d|r(x_h)-r(x_{h-1})|$ for the path $x=x_0,x_1,\dots,x_d=y$ in the Hasse diagram of $T$ and so $|F\cap G|=|F|-\frac{1}{2}|F\triangle G|$ is at least $\lfloor \frac{n-k}{2}\rfloor+i-\frac{1}{2} \sum_{h=1}^d|r(x_h)-r(x_{h-1})|\ge \min L_{T,n,k}(i)$.
\end{proof}

\medskip

For upward, downward, and two-way monotone trees it is even simpler to determine $m_{T,k}=m^*_{T,k}$.

\begin{prop}\label{maxexponent} \
    \begin{enumerate}
        \item 
        For any upward monotone tree $T$ and integer $k\ge h(T)$, we have $m_{T,k}=m_{T,k,k}=k-f(y)$, where $y\in T$ is the largest element of $T$ that is comparable with all elements and $f(y)$ is the length of the longest chain in $T$ with $y$ being its largest element.
        \item 
        For any downward monotone tree $T$ and integer $k\ge h(T)$, we have $m_{T,k}=m_{T,k,1}=f(z)-1$, where $z\in T$ is the smallest element of $T$ that is comparable with all elements and $f(z)$ is the length of the longest chain in $T$ with $z$ being its smallest element.
        \item 
        For any two-way monotone tree $T$ and integer $k\ge h(T)$, we have $m_{T,k}=m_{T,k,1}$ or $m_{T,k}=m_{T,k,k}$. Furthermore, if $m_{T,k}=m_{T,k,k}$ and $x$ is the largest element comparable to all other elements of $T$, then $m_{T,k}=m_{U_T[x],k-f(x)+1}=k-f(x)$.
    \end{enumerate}
\end{prop}

\begin{proof}
    By Observation \ref{layer}, $m_{T,k,i}$ is the maximum of $\frac{1}{2}\sum_{h=1}^d|r(x_h)-r(x_{h-1})|$ for some rank function $r:T\rightarrow [1,k]$ and elements $x_0,x_d$ with $r(x_0)=r(x_d)=i$ and the $x_h$ being the vertices of the path in $H(T)$ connecting $x_0,x_d$. In an upward monotone tree, this is just $i-r(a)$, where $a$ is the largest element in $D_T(x_0)\cap D_T(x_d)$. Clearly, $i$ cannot be more than $k$ and $a\ge y$, the largest element comparable to all $t\in T$. By definition, $r(y)\ge f(y)$ for all rank functions. This shows $m_{T,k}\le k-f(y)$. To see $m_{T,k}\ge k-f(y)$ consider the rank function $r$ with $r(x)=f(x)$ for all non-leaf element of $T$ and $r(x)=k$ for all leaves of $T$. Then there are two leaves $x,x'$ whose largest common ancestor is $y$, and so $\frac{1}{2}(|r(x)-r(y)|+|r(y)-r(x')|)=k-f(y)$.
    This finishes the proof of (1).

    (2) follows as the dual $T^\perp$ of a downward monotone tree poset $T$ is upward monotone. Applying (1) to $T^\perp$, we obtain (2).

    Finally, to see (3), observe that a two-way monotone tree $T$ is of the form $D_T[x_0]\cup \{x_1,\dots,x_{h-1}\}\cup U_T[x_h]$ for the chain $x_0\prec x_1\prec \dots \prec x_h$ of elements that are comparable to all elements of $T$. This implies that for any rank function $r:T\rightarrow [1,k]$ if $r(u)=r(v)$, then $u,v\in U_T[x_h]$ or $u,v\in D_T[x_0]$. Moreover, one can select a maximum chain $C$ in $D_T[x_0]$ and apply part (1) to $C\cup \{x_1,\dots,x_{h-1}\}\cup U_T[x_h]$ to obtain the largest $m_{T,k,i}$ because of the $U_T[x_h]$ part is $m_{T,k,k}$. Similarly, one can choose a maximum chain $C'$ in $U_T[x_h]$ and apply part (2) to $C'\cup \{x_1,\dots,x_{h-1}\}\cup D_T[x_0]$ to obtain the largest $m_{T,k,i}$ because of the $D_T[x_0]$ part is $m_{T,k,1}$.
\end{proof}

\begin{definition}
    For a family $\cF\subseteq 2^{[n]}$, its \textit{Lubell mass} is $$\lambda_n(\cF)=\sum_{F\in \cF}\frac{1}{\binom{n}{|F|}}=\frac{1}{n!}\sum_{F\in \cF}|F|!(n-|F|)!$$
    the expected value of $|\cC\cap \cF|$ over all maximal chains $\cC$ in $2^{[n]}$ taken uniformly at random. For convenience, we denote $\lambda_n(\cF)$ as $\lambda(\cF)$.

    For a set $F\in\cF$, we will write $\lambda_{n-|F|}(\cF)$ to denote the Lubell mass of $\cF(F)=\{F'\setminus F: F'\in\cF, F\subseteq F'\}$ in $2^{[n]\setminus F}$. 
\end{definition}

Next, we prove a special case of Theorem \ref{blowup} (1), from which the general case will easily follow.

\begin{thm}\label{blowupchain}
    For the chain poset $C_h$ on $h$ vertices and $x\in C_h$ and $\varepsilon>0$ there exists $\delta >0$ such that for $n$ large enough any family $\cF\subseteq 2^{[n]}$ with $|\cF|\ge (h-1+\varepsilon)\binom{n}{\lfloor \frac{n}{2}\rfloor}$ contains a weak copy of the blow-up $C_h(x,\delta n)$.
\end{thm}

Some of the ideas of the next proof are from \cite{GNPV}.

\begin{proof}
    Observe first that it is enough to prove the statement for $x$ being the minimal element of the chain. Indeed, by taking the complement family $\overline{\cF}=\{[n]\setminus F:F\in \cF\}$, we would immediately obtain the result for the maximal element of the chain. Finally, if $x$ is the $j$th element of the chain, then there exist $\delta_1$ belonging to $\varepsilon/3$ and the maximal element of $C_j$, and $\delta_2$ belonging to $\varepsilon/3$ and the minimal element of $C_{h-j+1}$. Setting $\delta=\min\{\delta_1,\delta_2\}$, for any family $\cF$ with $|\cF|\ge (h-1+\varepsilon)\binom{n}{\lfloor n/2\rfloor}$, one can define $\cU\subset \cF$ to be the subfamily of those sets that are not smallest elements of any weak copy of $C_{h-j+1}(x_{min},\delta n)$ in $\cF$ and $\cD\subset \cF$ to be the subfamily of those sets that are not largest elements of any weak copy of $C_{j}(x_{max},\delta n)$ in $\cF$. If the statement is proved for minimal and maximal elements of chains, then $|\cU|\le (h-j+\varepsilon/3)\binom{n}{\lfloor n/2\rfloor}$ and $|\cD|\le (j-1+\varepsilon/3)\binom{n}{\lfloor n/2\rfloor}$ and so any element in $\cF\setminus (\cU\cup \cD)$ is the root of a blowup $C_h(x_j,\delta n)$, and there are at least $\varepsilon/3\binom{n}{\lfloor n/2\rfloor}$ such sets.

    \medskip

    To prove the statement of the theorem with the minimal element of the chain being the root, we define families $\cF_h\supset \cF_{h-1}\supset \dots \supset \cF_1$ as follows. We set $\cF_h=\{F\in \cF: ||F|-n/2|\le n^{2/3}\}$, and then once $\cF_{j+1}$ is defined, we let $\cF_j=\{F\in \cF_{j+1}:\lambda_{n-|F|}(\cF_{j+1})\ge 1+\frac{\varepsilon}{2h}\}$. As shown in \cite{GNPV}:
    \begin{itemize}
        \item 
        $\cF\setminus \cF_h\subseteq \{G\subseteq [n]: ||G|-n/2|\ge n^{2/3}\}$ and this latter by Chernoff's inequality is of size $o(\frac{1}{n^2}\binom{n}{\lfloor n/2\rfloor})$ and so $|\cF_h|\ge (h-1+\varepsilon/2)\binom{n}{\lfloor n/2\rfloor}$,
        \item 
        for any $1\le j\le h-1$, the family $\cG_{j+1}=\cF_{j+1}\setminus \cF_j$, we have $\lambda(\cG_{j+1})\le 1+\varepsilon/2h$ and thus $|\cG_{j+1}|\le (1+\frac{\varepsilon}{2h})\binom{n}{\lfloor n/2\rfloor}$. Indeed, one uses the min-partition of the set $\mathbf{C}_n$ of all maximal chains of $2^{[n]}$ into $\cup_{G\in \cG_{j+1}}\mathbf{C}_G\cup \mathbf{C}_\emptyset$, where $\mathbf{C}_G$ denotes the set of all maximal chains $\cC$ such that the smallest set in $\cC\cap \cG_{j+1}$ is $G$ and $\mathbf{C}_\emptyset$ denotes the set of all maximal chains $\cC$ such that  $\cC\cap \cG_{j+1}$ is empty. By the definition of $\cF_j$ and $\cG_{j+1}$, for all $G\in \cG_{j+1}$ we have $\lambda_{n-|G|}(\cG_{j+1})\le 1+\frac{\varepsilon}{2h}$, and as $\lambda(\cG_{j+1})$ is a weighted average of the values $\lambda_{n-|G|}(\cG_{j+1})$, we have $\lambda(\cG_{j+1})\le 1+\frac{\varepsilon}{2h}$ as claimed.

        In particular, $|\cF_1|=|\cF_h|-\sum_{j=1}^{h-1}|\cG_{j+1}|\ge \frac{\varepsilon}{2h}\binom{n}{\lfloor n/2\rfloor}$, so $\cF_1$ is non-empty.
    \end{itemize}
    For a set $F\in \cF_j$ with $j<h$, we say that $F$ is of \textit{type $a$} for some integer $1\le a \le z:=h5^{h}$ if $\cF_{j+1}$ contains at least $\frac{\varepsilon}{3hz}\binom{n-|F|}{a}$ sets $F'$ with $F\subset F'$ and $|F'|=|F|+a$ and there is no $a'<a$ with the same property. Also, $F\in \cF_j$ is of \textit{type $+$} if it is not of type $a$ for any $a\le z$ and there exists $b> z$ such that $\cF_{j+1}$ contains at least $\frac{1}{n^2}\binom{n-|F|}{b}$ sets $F'$ with $F\subset F'$ and $|F'|=|F|+b$. Observe that every set $F\in \cF_j$ has a type as otherwise we would have $\lambda_{n-|F|}(\cF_{j+1})\le 1+z\cdot \frac{\varepsilon}{3hz}+2n^{2/3}\frac{1}{n^2}\le 1+\frac{\varepsilon}{2h}$ contradicting the definition of $F\in \cF_j$.

    \medskip

    Now we start defining the embedding of $C_h(x_{min},\delta n)$ into $\cF$. We do it level-by-level, with the root being an arbitrary set $F_{root}$ of $\cF_1$ and on the $j$th level we will use sets of $\cF_j$. We will keep the following rules:
    \begin{enumerate}
        \item 
        For any $j<h$, all sets on level $j$ will be of the same type.
        \item 
        Suppose there has been a level at which type + sets have been used, and the last level at which type + sets were used is level $j_0$. Then if $t,t'\in C_h(x_{min},\delta n)$ are at level $j>j_0$ and their ancestors at level $j_0$ are different, then if $F$ and $F'$ play the role of $t$ and $t'$ and $|F|=|F'|$, then $|F\triangle F'|> h5^{h-j}=:t_h(j)$.
    \end{enumerate}
    Suppose we have managed to define an embedding until level $j-1$ with (1) and (2) satisfied. Then to extend this embedding to level $j$, we first use a case analysis to obtain an embedding satisfying only (2), and then we apply a thinning out process that is the same in all cases to make the embedding satisfy (1).

    \medskip 
    
    \textsc{Case I} Sets at level $j-1$ are of type $a$ for $1\le a \le z$.

\smallskip

    \textsc{Case I/I} There is no level $i<j$ at which sets of type $+$ are used.

    Suppose all sets at level $j-1$ are of type $a$. Then for any set $F$ at level $j-1$, we pick $\frac{\varepsilon}{4hz}n$ supersets of $F$ of size $|F|+a$ in $\cF_j$. The same superset $F^*$ might occur for $F,F'\in \cF_{j-1}$, but, as $|F^*\setminus F_{root}|\le hz$, for any $F^*$ there are at most $D:=2^{hz}$ (a big but fixed constant) sets in $\cF_{j-1}$ it could have been picked for. So for each $F^*$ we pick uniformly at random to which of its subset at level $j-1$ it should belong. By Chernoff's bound the expected number of sets at level $j-1$ for which we have not picked at least $\frac{\varepsilon}{6hzD}n$ supersets $F^*$ tends to 0, so there exists a way to distribute the $F^*$ so that every set at level $j-1$ has at least $\frac{\varepsilon}{8hzD}n$ supersets. 

    \smallskip

    \textsc{Case I/II} There exists a level $i<j$ at which type $+$ sets are picked. Let the highest such level be $j_0<j$.

    \smallskip

    \textsc{Case I/II/I} $a\le 2t_h(j)$.

    For any set $F$ at level $j-1$, we again pick $\frac{\varepsilon}{4hz}n$ supersets of $F$ of size $|F|+a$ in $\cF_j$. If sets $F,F'\in \cF_{j-1}$ have different ancestors at level $j_0$, then by assumption (2), we have $|F\triangle F'|>t_h(j-1)$. So for any picked pair of supersets $F\subset F^*$, $F'\subset F'^*$, we have 
    $$|F^*\triangle F'^*|\ge |F\triangle F'|-2a > t_h(j-1)-4t_h(j)>t_h(j),$$ 
    so (2) is satisfied. If the ancestors of $F,F'$ at level $j_0$ are the same, then a superset $F^*$ can be assigned to both of them, but as in Case I/I we can distribute the $F^*$ so that each $F$ at level $j-1$ gets at least $\frac{\varepsilon}{8hzD}n$ supersets.
    
    \smallskip

    \textsc{Case I/II/II} $a> 2t_h(j)$.

    We pick supersets $F^*$ of sets $F$ at level $j-1$ greedily. There are less than $n^j\le n^h$ sets to be picked at level $j$. For each such set $F^*$ there are at most $n^{t_h(j)}$ other sets $F'^*$ of the same size with $|F^*\triangle F'^*|\le t_h(j)$. So the total number of too close sets is at most $n^{h+t_h(j)}=O(n^{2t_h(j)})=o(n^a)$. So we can pick supersets satisfying (2) even for those sets that have the same ancestor at level $j_0$. 

    \medskip 

    \textsc{Case II} Sets at level $j-1$ are of type $+$.

    Then for each set $F$ at level $j-1$, there exist some $b>z=h5^{h}$ such that $F$ has at least $\Omega(n^{z-2})$ supersets $F^*$ of size $|F|+b$. We want to embed $(\delta n)^{j-1}\le (\delta n)^h$ elements. For any $F^*$, the number of sets $F^{**}$ with $|F^*|=|F^{**}|$ and $|F^*\triangle F^{**}|\le t_h(j)$ is at most $n^{t_h(j)}$. So as $n^{z-2}=\omega(n^{h+t_h(j)})$ we can guarantee (2) greedily. Supersets of the same set $F$ at level $j$ will be distinct and there is no assumption on the size of their symmetric difference.

    \medskip

    Finally, we show how to thin out the already defined part of $C_h(x_{min}, \delta n)$ at the price of decreasing $\delta$ so that property (1) should be satisfied after adding level $j$ of the blow-up. So assume we have managed to define an embedding satisfying (2) as above and satisfying (1) up to level $j-1$. We now introduce a labeling that might differ from the type of a set $F$ in the blow-up, but is related to the types of $F$'s descendants. At level $j$, we label all sets with their type. Then at level $j-1$ we label all sets with the plurality type (the one that occurs the most number of times) of their children, and so on: once the labels of level $i$ are defined, then the label of a set $F$ at level $i-1$ is the label that appears the most number of times at its children at level $i$. Once we reach $F_{root}$ and define its label $\ell$, then we keep $F_{root}$ and its children with label $\ell$, and the children of these children labelled $\ell$, and so on. As the number of possible labels is $z+1$, all sets in the blow-up that survive the thinning will have at least $\frac{1}{z+1}$ of their original children. As the thinning out process occurs $h-2$ times, and originally, we define $\frac{\varepsilon}{8hzD}n$ children for all sets, at the end, we will still have a weak copy of $C_h(x_{min},\delta n)$, where $\delta=\frac{\varepsilon}{8hzD(z+1)^h}$. 
\end{proof}

Observe that for any two-way monotone tree poset $T$ and $y\in T$ that is comparable to all other elements of $T$, it holds that $r(y)$ is the same for any rank function $r:T\rightarrow [1,h(T)]$. So $T(y,\frac{\delta}{|T|}n)\subseteq C_{h(T)}(x_{f(y)},\delta n)$ and Theorem \ref{blowup} (1) follows from Theorem \ref{blowupchain}.

\smallskip

Next we prove Theorem \ref{blowup} (2) that we restate separately.

\begin{thm}\label{blowupheight2}
    For any tree poset $T$ of height $2$ and $x\in T$ and $\varepsilon>0$ there exists $\delta >0$ such that for $n$ large enough any family $\cF\subseteq 2^{[n]}$ with $|\cF|\ge (1+\varepsilon)\binom{n}{\lfloor \frac{n}{2}\rfloor}$ contains a weak copy $T(x,\delta n)$.
\end{thm}

\begin{proof}
    First we can assume that $|\mathcal{F}|\leq (1+\varepsilon')\binom{n}{\lfloor\frac{n}{2}\rfloor}$ for some constant $\varepsilon'>0$, and we claim that it is enough to prove the statement for path posets $P_k$ and their end vertices. Let the $k$ elements of $P_k$ be $a_1,a_2,\dots,a_{\lceil k/2\rceil},b_1,b_2,\dots,b_{\lfloor k/2\rfloor}$ with $a_i\prec b_i$ for $1\le i\le \lfloor\frac{k}{2}\rfloor$ and $a_{i+1}\prec b_i$ for $1\le i\le \lceil\frac{k}{2}\rceil-1$. If we know that any family of size $(1+\varepsilon)\binom{n}{\lfloor n/2\rfloor}$ contains a weak copy of $P_k(a_1,\delta n)$, then applying this to the family of complements $\bar{\cF}=\{\bar{F}:F\in \cF\}$, we obtain that every family of size $(1+\varepsilon)\binom{n}{\lfloor n/2\rfloor}$ contains a weak copy of $P^\perp_k(c_1,\delta n)$, where $P^\perp_k$ has $k$ vertices $c_1,c_2,\dots,c_{\lceil k/2\rceil},d_1,d_2,\dots,d_{\lfloor k/2\rfloor}$  with $d_i\prec c_i$ for $1\le i\le \lfloor\frac{k}{2}\rfloor$ and $d_{i}\prec c_{i+1}$ for $1\le i\le \lceil\frac{k}{2}\rceil-1$. Finally, if $T$ is an arbitrary poset of height 2 and $x\in T$, then, writing $r=r(x)=\max\{d(x,t):t\in T\}$ with $d(x,t)$ being the graph distance in the undirected Hasse diagram of $T$, then $T(x,\frac{\delta}{|T|}n)$ is a subposet of $P_{r+1}(a_1,\delta n)$ or $P^\perp_{r+1}(c_1,\delta n)$ depending on whether $x$ is a minimal or a maximal element of $T$.

    To prove the statement for $P_k$ and $a_1$, we set $\cF_0=\{F\in \cF: ||F|-n/2|\le n^{2/3}\}$, and  we let $\cD=\{F\in \cF_{0}:\lambda_{n-|F|}(\cF_{0})\ge 1+\frac{\varepsilon}{4}\}$. We have:
    \begin{itemize}
        \item 
        $\cF\setminus \cF_0\subseteq \{G\subseteq [n]: ||G|-n/2|\ge n^{2/3}\}$ and this latter by Chernoff's inequality is of size $o(\frac{1}{n^2}\binom{n}{\lfloor n/2\rfloor})$ and so $|\cF_0|\ge (1+\varepsilon/2)\binom{n}{\lfloor n/2\rfloor}$,
        \item 
        for the family $\cU=\cF_{0}\setminus \cD$, we have $\lambda(\cU)\le 1+\varepsilon/4$ and thus $|\cU|\le (1+\frac{\varepsilon}{4})\binom{n}{\lfloor n/2\rfloor}$. Indeed, one uses the min-partition of the set $\mathbf{C}_n$ of all maximal chains of $2^{[n]}$ into $\cup_{G\in \cU}\mathbf{C}_G\cup \mathbf{C}_\emptyset$, where $\mathbf{C}_G$ denotes the set of all maximal chains $\cC$ such that the smallest set in $\cC\cap \cU$ is $G$ and $\mathbf{C}_\emptyset$ denotes the set of all maximal chains $\cC$ such that  $\cC\cap \cU$ is empty. By definition of $\cD$ and $\cU$, for all $G\in \cU$ we have $\lambda_{n-|G|}(\cU)\le 1+\frac{\varepsilon}{4}$, and as $\lambda(\cU)$ is a weighted average of the values $\lambda_{n-|G|}(\cU)$, we have $\lambda(\cU)\le 1+\frac{\varepsilon}{4}$ as claimed.
        In particular, $|\cD|=|\cF_0|-|\cU|\ge \frac{\varepsilon}{4}\binom{n}{\lfloor n/2\rfloor}$.
    \end{itemize}
We say that a set $F \in \cD$ is of type $j$ for some $1\le j \le k+2$ if there exist $\frac{\varepsilon}{8k}\binom{n-|F|}{j}$ sets $F'\in \cF_0$ of size $|F|+j$ with $F\subset F'$ and no $j'<j$ exists with the same property. We say that a set $F \in \cD$ is of type $+$ if it is not of type $j$ for any $j$, but there exists $h\ge k+3$ such that at least $\frac{1}{n}\binom{n-|F|}{h}$ sets $F'\in \cF_0$ of size $|F|+h$ are supersets of $F$.

We claim that every set $F\in \cD$ has a type. Indeed, if not then $\lambda_{n-|F|}(\cF_0)\le 1+\frac{\varepsilon(k+2)}{8k}+\frac{2n^{2/3}}{n}<1+\frac{\varepsilon}{4}$ for large enough $n$ and $k\ge 3$. (This latter we can assume as the statement for any $k$ implies the statement for any $k'<k$.) So there is a type such that the number of sets in that type is at least $\frac{\varepsilon}{4(k+3)}\binom{n}{\lfloor n/2\rfloor}$. Let $\cD_j$ and $\cD_+$ denote the sets of type $j$ and $+$, respectively.

\medskip

\textsc{Case I} $|\cD_j|\ge \frac{\varepsilon}{4(k+3)}\binom{n}{\lfloor n/2\rfloor}$ for some $1\le j\le k+2$.

\smallskip

Let us introduce the following bipartite graph $B_j$ with parts $\cD_j$ and $\cF_0$ such that $FF'$ with $F\in \cD_j$ and $F'\in \cF_0$ is an edge if and only $F\subset F'$ and $|F'|=|F|+j$. The average degree in $B_j$ is at least $\frac{\varepsilon^2}{64k(k+3)}\binom{n/2-2n^{2/3}}{j}$, and thus there exists a subgraph $B_j'$ of $B_j$ with minimum degree at least $\frac{\varepsilon^2}{128k(k+3)}\binom{n/2-2n^{2/3}}{j}$. We shall embed $P_k(a_1,\delta n)$ for some small enough $\delta$ into the sets corresponding to the vertices of $B_j'$. We choose an arbitrary set $F$ in the $\cD_j$ part of $B_j'$. For any edge $F_1F_1'$ of $B_j'$, we say that its colors are the elements of $F_1\triangle F_1'$. We want an embedding of $P_k(a_1,\delta n)$ such that for any $F'$ that we use in the embedding the color sets of the edges on the path from $F'$ to $F$ are pairwise disjoint. Observe that if we embedded $P_k(a_1,\delta n)$ until vertices that are at distance at most $d$ from $a_1$, then the number of colors used is $dj$ and the number of possible neighbors of a current set $F_0$ using at least one of these colors is at most $dj\binom{n/2+n^{2/3}}{j-1}=o(\binom{n/2-n^{2/3}}{j})$ as $j$ is fixed. As $j\ge 1$, we do have at least $\frac{\varepsilon^2}{260k(k+3)} n$ neighbors of $F_0$ that would not use colors from previous edges of the path from $F_0$ to $F$. Unfortunately, the same set $F_{00}$ might be assigned to different $F_0$ and $F_0'$, but, by the assumption that the color sets are pairwise disjoint along a path, we have that $F_{00}\triangle F \supset F_0\triangle F$ and  $F_{00}\triangle F \supset F_0'\triangle F$. As $|F_{00}\triangle F|\le (k-1)j$, the number of possible parents for $F_{00}$ is at most $2^{(k-1)j}$. So for every  $F_{00}$ we assign a parent uniformly at random over all possible parents. By Chernoff's inequality the probability that a parent is assigned to less than half of the expected number of children is less than $e^{-\frac{\varepsilon^4}{(260k(k+3))^2}n}$. As there are at most $n^{k-2}$ possible parents, with probability tending to one, all parents get at least $\frac{\varepsilon^2}{520k(k+3)2^{kj}} n$ children. This shows that we can continue this process to embed $P_k(a_1,\delta n)$ for small enough $\delta$.

\medskip

\textsc{Case II} $|\cD_+|\ge \frac{\varepsilon}{4(k+3)}\binom{n}{\lfloor n/2\rfloor}$.

\smallskip

Then there exists $k+3\le h \le 2n^{2/3}$ such that there exists $\cD_h\subseteq \cD_+$ with all $F\in \cD_h$ having at least $\frac{1}{n}\binom{n-|F|}{h}$ supersets $F'\in \cF_0$ with $|F'|=|F|+h$ and $|\cD_h|\ge \frac{\varepsilon}{8(k+3)n^{2/3}}\binom{n}{\lfloor n/2\rfloor}$. We consider the bipartite graph $B_h$ with parts $\cD_h$ and $\cF_0$ and edges defined by containment. The average degree is at least $\frac{\varepsilon}{16(k+3)n^{2/3}}\frac{1}{n}\binom{n/2-2n^{2/3}}{h}=\Omega(n^{k+1/3})$, and thus there exists  a subgraph $B_h'$ of $B_h$ with minimum degree $\Omega(n^{k+1/3})$. Now we can embed $P_k(a_1,\delta n)$ greedily according to $B_h'$ as degrees are $\Omega(n^{k+1/3})$ and $|P_k(a_1,\delta n)|=O(n^{k-1})$, so at any moment a set corresponding to a vertex of $B_h'$ has at least $n$ more neighbors than already embedded sets. 
\end{proof}

\begin{lem}\label{blowuppembedtwoway}
    Let $T$ be a branching two-way monotone tree poset of height $h(T)$ and suppose that for some $k\ge h(T)$ we have $m_{T,k}=m_{T,k,k}$. Let $x$ be the element of $T$ that is comparable to all elements of $T$ and among such elements is the largest one and let $r=f(x)$. Then for the chain $C_k=\{x_1\prec x_2\prec \dots \prec x_{k}\}$, we have that for any pair $y,y'$ of maximal elements in $C_k(x_r,\delta n)$ there exists a weak copy of $T$ in $C_k(x_r,\delta n)$ that contains $y,y'$.
\end{lem}

We will use the following notions in the proof: a \textit{cover of $a$} is an element $a'$ with $a\prec a'$ and no $b\neq a,a'$ with $a\prec b\prec a'$. In an upward monotone tree, a \textit{branch starting at $z$} is a sequence of $z=x_0,x_1,\dots,x_\ell$ such that $x_i$ is a cover of $x_{i-1}$ for all $1\le i\le \ell$.

\begin{proof}
    The assumptions $m_{T,k}=m_{T,k,k}$ imply $h(D_T(x))\leq h(U_T(x))$
    If $h(D_T(x))<h(U_T(x))$, the fact that $T$ is branching implies that $U_T(x)$ is branching. If $h(D_T(x))=h(U_T(x))$, we have $m_{T,k}=m_{T,k,k}=m_{T,k,1}$ and at least one of $D_T(x)$ and $U_T(x)$ is branching. Without loss of generality, we can assume that $U_T(x)$ is branching.  We embed the $D_T(x)$ part of $T$ such that the image of $x$ is $x_r$. Let the largest ranked common predecessor of $y,y'$ be $z$.  We have $r\le r(z)\le k-1$. Let $z'$ be such that $x_r\preceq z'\preceq z$ and $k-r(z')\ge h(T)-r$ and $r(z)-r(z')\le h(T)-1-r$. Such $z'$ exists as  these conditions can be rewritten as $r(z)-h(T)+1+r\le r(z')\le k-h(T)+r$ and $r\le r(z')$. Since $h(T)\le k$, we have $r\le k-h(T)+r$, so if $r(z')$ can be picked to satisfy $r(z)-h(T)+1+r\le r(z')\le k-h(T)+r$, then $r\le r(z')$ can also be satisfied. But $r(z')$ can be picked to satisfy $r(z)-h(T)+1+r\le r(z')\le k-h(T)+r$ as $r(z)\le k-1$ since $z$ is not top ranked. Once the value of $r(z')$ is fixed, one can pick the element $z'$ as every intermediate rank appears on the chain from $x_r$ to $z$.
    
     Now as $U_T(x)$ is branching, there is a copy $U$ of $U_T(x)$ with $x$ played by $z'$ and $z\in U$ such that, by the branching property of $U_T(x)$, $z$ has at least two covers in $U$. These can be chosen such that one branch of $U$ starting at $z$ follows the branch of $C_k(x_r,\delta n)$ starting at $z$ and leading to $y$, while another follows the branch to $y'$. If these do not reach $y$ and $y'$, then we replace the leaves on these branches by $y$ and $y'$. Finally, we remove $z'$ unless $z'=x_r$ so that $x$ is not played by two distinct elements. As if $x_r\neq z'$, then $x_r\prec z'$ and $z'$ is the smallest element in $U$, $x_r\prec u$ for all $u\in U$. So the remaining elements will form a copy of $T$ containing $y,y'$. 
\end{proof}

For a tree poset $T$ of height 2, we define the radius $rad(T)=\min_{u\in T}\max_{v\in T}d_{H(T)}(u,v),$ where $d_{H(T)}$ denotes the graph distance in the undirected  Hasse  diagram of $T$, and call every vertex attaining this minimum a center of $T$.
\begin{lem}\label{blowupembedheight2}
    Let $T$ be a tree poset of height 2 of radius $r$ and center $x$, and let $y,y'\in P_{r+1}(a,\delta n)$ be elements, where $a$ is an endpoint of $H(P_{r+1})$, such that $x$ and $a$ are either both maximal or both minimal elements of their posets. 
    Assume that the shortest path from $y$ to $y'$ contains $a$ and $d(y,y')$ is at most the diameter of $T$. Then there exists a weak copy $T^*$ of $T$ with $P_{r+1}(a,\delta n)\supset T(x,\frac{\delta}{|T|}n)\supset T^*\ni y,y'$.
\end{lem}

\begin{proof}
    The existence of $T(x,\frac{\delta}{|T|}n)$ follows from the fact that every element of $T$ can have at most $|T|$ neighbors in $H(T)$ and so by partitioning the neighbors of an element $u$ in $P:=P_{r+1}(a,\delta n)$ into at most $|T|$ parts we can get a copy of $T(x,\frac{\delta}{|T|}n)$. 

    Let $s=d_{H(P)}(y,a)$ and $t=d_{H(P)}(y',a)$,
and assume without loss of generality that $s\ge t$. Since the
shortest path from $y$ to $y'$ contains $a$, we have
\[
s+t=d_{H(P)}(y,y')\le \operatorname{diam}(T).
\]
Let $D=\operatorname{diam}(T)$. A diameter path of $H(T)$ passing
through the center $x$ has two parts, starting at $x$, of lengths
$\left\lceil\frac D2\right\rceil$ and $\left\lfloor\frac D2\right\rfloor$.
We claim that, after possibly interchanging the two parts, they
contain vertices $t_1,t_2$ at distances $s,t$, respectively, from
$x$. Indeed, if $s=\lceil\frac{D}{2}\rceil$, then
$t\le \lfloor \frac{D}{2}\rfloor$.
If $s<\lceil\frac{D}{2}\rceil$, then
$s,t\le =\lfloor\frac{D}{2}\rfloor$.
Thus such $t_1,t_2$ exist, and the path from $t_1$ to $t_2$ passes through $x$.

We map $x,t_1,t_2$ to $a,y,y'$, respectively. We then extend this
map recursively to a weak embedding of $T$ by partitioning, at each
step, the available children in the blow-up into at most $|T|$
parts. Since each vertex of $T$ has degree at most $|T|$, every
required part has size at least $\delta n/|T|$. This yields a weak
copy of $T$ containing $y$ and $y'$ and, more generally, a copy of
$T(x,\frac{\delta}{|T|} n)$ containing the prescribed path.
\end{proof}

\begin{proof}[Proof of Theorem \ref{middletree}]
    To see (1), let $T$ be a two-way monotone branching tree poset and let $\cF\subseteq 2^{[n]}$ be a family of size $(k-1+\varepsilon)\binom{n}{\lfloor n/2\rfloor}$. Without loss of generality we can assume $m_{T,k}=m_{T,k,k}$ otherwise we can consider the dual poset $T^\perp$ and the family of complements $\bar{\cF}=\{[n]\setminus F:F\in \cF\}$. Let $y\in T$ be the largest element among those that are comparable to all elements of $T$. By Theorem \ref{blowupchain}, $\cF$ contains a weak copy $\cT$ of $C_k(x_{f(y)},\delta n)$. According to Lemma \ref{blowuppembedtwoway}, the sets of $\cT$ corresponding to top ranked elements of $C_k(x_{f(y)},\delta n)$ form a clique in $G_T(\cT)\subseteq G_T(\cF)$ and their size is $(\delta n)^{k-f(y)}$. But according to Proposition \ref{maxexponent}, we have $m_{T,k}=m_{T,k,k}=k-f(y)$. This finishes the proof of (1).

    To see (2), observe that for any tree poset $T$ of height 2, there is only one rank function $r:T\rightarrow \{1,2\}$, namely $r(x)=1$ for all minimal elements $x$ and $r(y)=2$ for all maximal elements $y$. Therefore, $m_{T,2,1}=\frac{1}{2}\max\{d_{H(T)}(x,x'):x,x' ~\text{are minimal}\}, m_{T,2,2}=\frac{1}{2}\max\{d_{H(T)}(y,y'):y,y' ~\text{are maximal}\}$.

    Suppose that $m_{T,2,2}\le m_{T,2,1}$, and put
$m=m_{T,2,1}$.
The case $m=0$ is immediate, so assume that $m\ge1$. Choose two
minimal elements $x,x'\in T$ with
$d_{H(T)}(x,x')=2m$,
and let $z$ be the midpoint of the path from $x$ to $x'$.

We first observe that $\operatorname{diam}(T)\le 2m+1$.
Indeed, a diameter whose endpoints are both minimal has length at
most $2m$, while one whose endpoints are both maximal has length at
most $2m_{T,2,2}\le2m$. If the endpoints have different types, then
deleting either endpoint produces a same-type path of length one
less, and hence the diameter is at most $2m+1$. Consequently, $\operatorname{ecc}_{H(T)}(z)\le m+1$ where $\operatorname{ecc}_G(v)=\max\{d_G(u,v):u\in V(G)\}$.

Let $P_{2m+2}$ be a height-two path poset and let $a_0$ be one of
its endpoints, chosen so that $a_0$ and $z$ are either both minimal
or both maximal. By Theorem~\ref{blowupheight2}, $\cF$ contains a
weak copy $\cQ$ of $P_{2m+2}(a_0,\delta n)$.
Let $\cA$ be the family of elements of $\cQ$ corresponding to the
unique vertex of $P_{2m+2}$ at distance $m$ from $a_0$. This vertex
is minimal, and $|\cA|=(\delta n)^m$.

We claim that $\cA$ is a clique in $G_T(\cQ)$. Let $F,F'\in\cA$
be distinct, and let $B$ be their last common ancestor in the
rooted blow-up. Write $s=d_{\cQ}(a_0,B)$ and $q=m-s$.
Then $q\ge1$, the path from $F$ to $F'$ passes through $b$, and
\[
d_{\cQ}(F,B)=d_{\cQ}(F',B)=q.
\]

If $B$ has the same type as $z$, put $w=z$. Otherwise, let $w$ be
a neighbor of $z$ on the path from $x$ to $x'$. In the latter case
$s\ge1$, and hence $q\le m-1$. Thus in both cases the two components
of the path from $x$ to $x'$ in $H(T)$ contain vertices $t,t'$ at
distance $q$ from $w$. Moreover, $t,t'$ are minimal, just as the
elements represented by $F,F'$.

The subtree of $P_{2m+2}(a_0,\delta n)$ rooted at $B$ has depth $2m+1-s=m+1+q$.
If $w=z$, then $\operatorname{ecc}_{H(T)}(w)\le m+1$. Otherwise,
\[
\operatorname{ecc}_{H(T)}(w)
 \le \operatorname{ecc}_{H(T)}(z)+1
 \le m+2
 \le m+1+q.
\]
We may therefore embed $T$ recursively into this rooted subtree,
mapping $w,t,t'$ to $B,F,F'$, respectively. Hence some weak copy of
$T$ in $\cQ$ contains both $F$ and $F'$.

It follows that $\cA$ is a clique in $G_T(\cQ)\subseteq G_T(\cF)$.
Therefore,
$\chi(G_T(\cF))
 \ge |\cA|
 =(\delta n)^m
 =\Omega(n^{m_{T,2}})$,
which proves part~(2).  
\end{proof}
\section*{Acknowledgement}
The research of Binlong Li was supported by the National Natural Science Foundation of China (Grant No. 12671417). 
Changxin Wang was supported by the China
Scholarship Council (No. 202506290206) and the National Natural Science Foundation of China (Grant No. 12471334).

\end{document}